\documentclass[12pt]{amsart}
\usepackage[english]{babel}
\usepackage{amsmath,amsthm,amssymb,amsfonts}
\usepackage[pdftex]{color}
\usepackage[dvipsnames]{xcolor}
\usepackage[bookmarks=true,hyperindex,pdftex,colorlinks,citecolor=blue,linkcolor=blue,urlcolor=blue]{hyperref}
\usepackage{graphicx}
\usepackage{mathtools}
\usepackage{mathabx}
\usepackage{cancel}
\usepackage{framed}
\usepackage{float}
\usepackage{tikz-cd}
\usepackage{indentfirst}

\newtheorem{Theorem}{Theorem}[section]

\newtheorem{Lemma}[Theorem]{Lemma}
\newtheorem{Corollary}[Theorem]{Corollary}

\newtheorem{corollary}[Theorem]{Corollary}
\newtheorem{proposition}[Theorem]{Proposition}

\newtheorem{theorem}[Theorem]{Theorem}
\newtheorem{lemma}[Theorem]{Lemma}

\theoremstyle{definition}
\newtheorem{Definition}[Theorem]{Definition}

\newtheorem{Remark}[Theorem]{Remark}

\newtheorem{remark}[Theorem]{Remark}

\newtheorem{example}[Theorem]{Example}

\DeclareMathOperator{\supp}{supp}

\DeclareMathOperator{\re}{Re}

\DeclareMathOperator{\diam}{diam}
\DeclareMathOperator{\co}{co}

\DeclareMathOperator{\Id}{Id}

\newcommand{\bbn}{\mathbb{N}}
\newcommand{\bbk}{\mathbb{K}}
\newcommand{\bbr}{\mathbb{R}}
\newcommand{\bbc}{\mathbb{C}}
\newcommand{\call}{\mathcal{L}}
\newcommand{\cala}{\mathcal{A}}

\renewcommand{\re}{\operatorname{Re}}

\newcommand{\conv}{\operatorname{conv}}
\renewcommand{\epsilon}{\varepsilon}
\newcommand{\ste}{\operatorname{St}_\varepsilon}

\newcommand{\id}{\operatorname{Id}}
\newcommand{\calb}{\mathcal{B}}

\begin{document}

\title[Diametral notions in the unit ball of some vector-valued function spaces]{On various diametral notions of points in the unit ball of some vector-valued function spaces}

\begin{abstract}
%\textcolor{red}{Abstract is a v1 draft, should be rewritten}

In this article, we study the ccs-Daugavet, ccs-$\Delta$, super-Daugavet, super-$\Delta$, Daugavet, $\Delta$, and $\nabla$ points in the unit balls of vector-valued function spaces $C_0(L, X)$, $A(K, X)$, $L_\infty(\mu, X)$, and $L_1(\mu, X)$. To partially or fully characterize these diametral points, we first provide improvements of several stability results under $\oplus_\infty$ and $\oplus_1$-sums shown in the literature. For complex Banach spaces, $\nabla$ points are identical to Daugavet points, and so the study of $\nabla$ points only makes sense when a Banach space is real. Consequently, we obtain that the seven notions of diametral points are equivalent for $L_\infty(\mu)$ and uniform algebra when $K$ is infinite.

%We study the following classes of points (related to the diameter-two properties) in the unit ball of Banach spaces: ccs-Daugavet, ccs-$\Delta$, super-Daugavet, super-$\Delta$, Daugavet, $\Delta$, and $\nabla$. We get some stability results for those notions under $\oplus_\infty$ and $\oplus_1$-sums of Banach spaces, and we use them to study these points in several classes of Banach spaces. We characterize $\Delta$, Daugavet, and $\nabla$ points in $C_0(L,X)$, $A(K,X)$, $L_\infty(\mu, X)$, and $L_1(\mu, X)$ spaces. We also study and partially or fully characterize some of the other notions in those spaces. In particular, we characterize all notions in uniform algebras $A(K)$, and complex ccs-$\Delta$ points in $L_1(\mu)$ spaces (which was the only open case left for these notions).
\end{abstract}

% We characterize Daugavet and $\nabla$ points in the vector-valued function spaces $C_0(L,X)$, $A(K,X)$, $L_\infty(\mu, X)$, and $L_1(\mu, X)$. We also study other diametral point notions in these spaces and provide some partial and full characterizations for some of them. Among our results, we show that in uniform algebras, the notions of ccs-Daugavet and $\Delta$ points coincide, and the same happens in all $C_0(L,X)$ and $A(K,X)$ spaces if $X$ is LUR. 

% \begin{abstract}
% We investigate various diametral notions of points on the unit ball of vector-valued function spaces $C_0(L, X)$, $A(K, X)$, and $L_p(\mu, X)$. Assuming $X$ is locally uniformly rotund, we provide characterizations of (ccs- and super-) Daugavet points and (ccs- and super-) $\Delta$ points as well as $\nabla$ points in $C_0(L, X)$ and $A(K, X)$. Under the same assumption, these characterizations show that the polynomial Daugavet property, Daugavet property, diametral diameter two properties, and the property ($D$) are equivalent in those spaces. We also provide a detailed analysis of the diametral points in $L_{p}(\mu, X)$, which exhibits a different behavior from their scalar analogues.    
% \end{abstract}

\thanks{
The first author was supported by Basic Science Research Program through the National Research Foundation of Korea (NRF) funded by the Ministry of Education, Science and Technology [NRF-2020R1A2C1A01010377].
}
\thanks{			
The second author was supported by Basic Science Research Program through the National Research Foundation of Korea (NRF) funded by the Ministry of Education, Science and Technology [NRF-2020R1A2C1A01010377], and also by Spanish MICIU / AEI / 10.13039 / 501100011033 and ERDF/EU through the grant PID2021-122126NB-C33.
}
\thanks{
The third author was supported by Basic Science Research Program through the National Research Foundation of Korea (NRF) funded by the Ministry of Education, Science and Technology [NRF-2020R1A2C1A01010377].
}
\author[Han Ju Lee]{Han Ju Lee}
\address[Han Ju Lee]{Department of Mathematics Education, Dongguk University, 04620 Seoul, Republic of Korea. 
\href{https://orcid.org/0000-0001-9523-2987}{ORCID: \texttt{0000-0001-9523-2987}}}
\email{\texttt{hanjulee@dgu.ac.kr}}

\author[\'Oscar Rold\'an]{\'Oscar Rold\'an}
\address[\'Oscar Rold\'an]{Departamento de An\'{a}lisis Matem\'{a}tico,
Universidad de Valencia, Doctor Moliner 50, 46100 Burjasot (Valencia), Spain.
\href{https://orcid.org/0000-0002-1966-1330}{ORCID: \texttt{0000-0002-1966-1330}}}
\email{oscar.roldan@uv.es}

\author[Hyung-Joon Tag]{Hyung-Joon Tag}
\address[Hyung-Joon Tag]{Department of Mathematics and Statistics, University of North Carolina at Greensboro, Greensboro, North Carolina 27402, USA. \href{https://orcid.org/0009-0005-1787-135X}{ORCID: \texttt{0009-0005-1787-135X}}}
\email{hjtag4@gmail.com}
\date{\today}

\maketitle

% \textcolor{red}{Abstract is once more obsolete, we now characterize many more notions.}

\section{Introduction}\label{section-introduction}

In 1963, I.K. Daugavet \cite{Daugavet63} showed that   every compact operator $T$ on  C([0,1]),  the Banach space of real-valued continuous functions on the unit interval $[0, 1]$, satisfies the equation
\begin{equation}\label{Daugavet-equation}
\|\Id+T\|=1+\|T\|.
\end{equation}
A Banach space $X$ is said to have the \textit{Daugavet property} if all rank-one operators on $X$ satisfy \eqref{Daugavet-equation}, and in this case all weakly compact operators also satisfy the same equation \cite[Theorem 2.3]{KSSW00}. The study of the Daugavet property and its consequences to the geometry of Banach spaces is a very active line of research nowadays. We refer to \cite{KSSW00,Werner01} for background. The Daugavet property is also closely related to various diameter two properties, which means that certain subsets of the unit ball of a Banach space $X$ have diameter two. These properties are at the opposite end of the spectrum from well-known properties such as the Radon-Nikod\'ym property or the point of continuity property. We refer to \cite[Sections 1 and 2]{MPZpre} for a detailed exposition of these concepts and their related results. 

In this article, we study classes of points on the unit sphere $S_X$ of a Banach space $X$ closely connected to the ``diametral'' (or pointwise) versions of the diameter two properties.

Let $X$ and $Y$ be Banach spaces over the scalar field $\bbk$,  a real or complex field. Denote by $B_X$, $S_X$, and $X^*$ the closed unit ball of $X$, the unit sphere of $X$, and the topological dual of $X$, respectively. Let $\call(X,Y)$ be the space of all linear and bounded operators from $X$ to $Y$. 
Given a bounded set $A\subset X$, a functional $x^*\in S_{X^*}$, and a positive $\varepsilon>0$, the corresponding \textit{slice} is the set
$$S(A, x^*, \varepsilon):=\{x\in A:\, \re\,x^*(x)>\sup_{y\in A} \re\,x^*(y)-\varepsilon\}.$$
When the set $A$ coincides with $B_X$, we shall just write $S(x^*, \varepsilon)$ if there is no possible confusion. A convex combination of slices of $A$ is a set $C=\sum_{i=1}\lambda_i S_i$, for some $n\in\bbn$, some $\lambda_1,\ldots,\lambda_n>0$ with $\sum_{i=1}^n \lambda_i=1$, and some slices $S_1, \ldots, S_n$ of $A$ defined as 
$$C:=\left\{x\in A:\, x=\sum_{i=1}^n \lambda_i x_i, \text{ where }x_i\in S_i\text{ for all }1\leq i \leq n\right\}.$$
%\textcolor{red}{ It's better to put the definition of convex combination of slices.}

\begin{Definition}
Let $X$ be a Banach space. Then
\begin{enumerate}
%\item $X$ has the Daugavet property if every rank-one operator $T: X \rightarrow X$ satisfies $\|I + T\|= 1 + \|T\|$.
\item $X$ is said to have the \textit{restricted diametral strong diameter two property} (\textit{restricted-DSD2P}) if for every convex combination of slices $C$ of $B_X$ and every $x \in C \cap S_X$, $\sup_{y \in C} \|x - y\| = 2$.
\item $X$ is said to have the \textit{diametral diameter two property} (\textit{DD2P}) if for every nonempty weakly open subset $W$ of $B_X$ and every $x \in W \cap S_X$, $\sup_{y \in W}\|x - y\| = 2$.
\item $X$ is said to have the \textit{diametral local diameter two property} (\textit{DLD2P}) if for every slice $S$ of $B_X$ and every $x \in S \cap S_X$, $\sup_{y \in W}\|x - y\| = 2$.
\item $X$ is said to have the \textit{property ($D$)} if every rank-one, norm-one projection $P:X \rightarrow X$ satisfies $\|I  - P\| = 2$.
\end{enumerate}
\end{Definition}

Note that if a Banach space satisfies one of these properties, then every slice of the unit ball has diameter two. Moreover, it is well known that the following implications hold.

% https://q.uiver.app/#q=WzAsNSxbMCwwLCJcXHRleHR7RGF1Z2F2ZXR9Il0sWzAsMSwiXFx0ZXh0e1Jlc3RyaWN0ZWQgRFNEMlB9Il0sWzEsMCwiXFx0ZXh0e0REMlB9Il0sWzIsMCwiXFx0ZXh0e0RMRDJQfSJdLFsyLDEsIlxcdGV4dHsoRCl9Il0sWzAsMSwiKGMpIl0sWzAsMiwiKGEpIl0sWzIsMywiKGIpIl0sWzMsNCwiKGUpIl0sWzEsMywiKGQpIiwyXV0=
\[\begin{tikzcd}
	{\text{Daugavet}} & {\text{DD2P}} & {\text{DLD2P}} \\
	{\text{Restricted DSD2P}} && {\text{(D)}}
	\arrow["{(a)}", from=1-1, to=1-2]
	\arrow["{(c)}", from=1-1, to=2-1]
	\arrow["{(b)}", from=1-2, to=1-3]
	\arrow["{(e)}", from=1-3, to=2-3]
	\arrow["{(d)}"', from=2-1, to=1-3]
\end{tikzcd}\]

The converse of implication $(a)$ as well as the converse of implication $(c)$ do not hold in general (see \cite{BLR18} and \cite[Section 4.3.2]{MPZpre}). We also mention that, to the best of the authors' knowledge, the converses of implications $(b)$, $(d)$, and $(e)$ remain unsolved.

In \cite{AHLP}, Daugavet and $\Delta$ points were introduced and studied in detail as localized versions of  the Daugavet property and DLD2P, respectively. The study of these notions has attracted a significant amount of researchers and provided interesting results and examples related to the diameter two properties. Furthermore, new notions related to these points were introduced and extensively studied in \cite{MPZpre}.

\begin{Definition}\label{def-six-notions}
Let $X$ be a Banach space.
\begin{enumerate}
\item A point $x \in S_X$ is called a \textit{ccs-Daugavet} point if $\sup_{y \in C}\|x - y\| = 2$ for every convex combination of slices $C$ of $B_X$.
\item A point $x \in S_X$ is called a \textit{super-Daugavet} point if $\sup_{y \in W}\|x - y\| = 2$ for every nonempty weakly open subset $W$ of $B_X$.
\item A point $x \in S_X$ is called a \textit{Daugavet} point if $\sup_{y \in S}\|x - y\| = 2$ for every slice $S$ of $B_X$.
\item A point $x \in S_X$ is called a \textit{ccs-$\Delta$} point if $\sup_{y \in C} \|x - y\| = 2$ for every convex combination $C$ of slices of $B_X$ that contains $x$.
\item A point $x \in S_X$ is a \textit{super-$\Delta$} point if $\sup_{y \in W}\|x - y\| = 2$ for every nonempty weakly open subset $W$ of $B_X$ that contains $x$.
\item A point $x \in S_X$ is a \textit{$\Delta$} point if $\sup_{y \in S}\|x - y\| = 2$ for every slice $S$ of $B_X$ that contains $x$.
\item A point $x \in S_X$ is a \textit{$\nabla$} point if $\sup_{y \in S}\|x - y\| = 2$ for every slice $S$ of $B_X$ that does not contain $x$.
\end{enumerate}
\end{Definition} 
\noindent We mention that $X$ has the Daugavet property, the restricted DSD2P, the DD2P, and the DLD2P if and only if every point in $S_X$ is a Daugavet point, a ccs-$\Delta$ point, a super-$\Delta$ point, and a $\Delta$ point, respectively \cite{MPZpre}. If $X$ has the Daugavet property, then every point in $S_X$ is a ccs-Daugavet point as a consequence of \cite[Lemma 3]{Shvydkoy00} and the results derived from its proof (see the discussion after \cite[Definition 2.5]{MPZpre}). From their definitions and the Bourgain's Lemma \cite[Lemma II.1]{GGMS87}, the following implications always hold for a point $x\in S_X$:

% \textcolor{red}{2 versions of diagram, flatter one, or tidier one:}

% \[\begin{tikzcd}
% 	{\text{ccs-Daugavet}} & {\text{super-Daugavet}} & {\text{Daugavet}} & \Delta \\
% 	{\text{ccs-}\Delta} & {\text{super-}\Delta}
% 	\arrow[from=1-1, to=1-2]
% 	\arrow[from=1-1, to=2-1]
% 	\arrow[from=1-2, to=1-3]
% 	\arrow[from=1-2, to=2-2]
% 	\arrow[from=1-3, to=1-4]
% 	\arrow[from=2-1, to=1-4]
% 	\arrow[from=2-2, to=1-4]
% \end{tikzcd}\]

\[\begin{tikzcd}
	{\text{ccs-Daugavet}} & {\text{super-Daugavet}} & {\text{Daugavet}} \\
	& {\text{super-}\Delta} \\
	{\text{ccs-}\Delta} && \Delta
	\arrow[from=1-1, to=1-2]
	\arrow[from=1-1, to=3-1]
	\arrow[from=1-2, to=1-3]
	\arrow[from=1-2, to=2-2]
	\arrow[from=1-3, to=3-3]
	\arrow[from=2-2, to=3-3]
	\arrow[from=3-1, to=3-3]
\end{tikzcd}\]

Clearly, every Daugavet point is a $\nabla$ point, and it is known that every $\nabla$ point is either a Daugavet point or a strongly exposed point (this is shown in \cite[Theorem 2.3]{HLPV} in the real case, and its proof can be easily adapted to the complex case). In \cite{HLPV}, $\nabla$ points are characterized for real-valued $C_0(L)$ and $L_1(\mu)$ spaces, and unlike with the other six notions, there are finite-dimensional real Banach spaces containing $\nabla$ points in the sphere. It is worth noting, however, that in complex spaces, $\nabla$ points and Daugavet points always coincide (see Proposition \ref{prop-nablas-complex}). We refer to \cite{CILMPQQRR24,JCpre,JR22,HLPV,MPZpre} for up-to-date information on all these notions and the relations between them.

Even though these concepts are not generally equivalent, some of them may coincide in certain Banach spaces. For instance, the six first notions from Definition \ref{def-six-notions} are equivalent in all isometric $L_1(\mu)$-preduals \cite[Corollary 4.5]{MPZpre}. In particular, this holds for $C(K)$ spaces, where such points can be characterized by the norm-attainment at limit points in $K$ \cite[Theorem 3.4]{AHLP}. In $L_1(\mu)$ spaces, the equivalence between Daugavet, $\Delta$, super-Duagavet, and super-$\Delta$ points is shown in \cite[Section 4.2]{MPZpre}, and these points are characterized by the nonexistence of an atom in their supports. In fact, these notions are also equivalent to ccs-$\Delta$ in these spaces (this was shown in \cite[Proposition 4.9]{MPZpre} in the real case, and we will see the complex case in Theorem \ref{theorem:ccs-delta-L1}). However, these notions are not equivalent to ccs-Daugavet in general for those spaces (see \cite[Proposition 4.12]{MPZpre}).

%In \cite[Section 4]{MPZpre}, the following was shown for instance.
%\begin{itemize}
%\item If $X=C(K)$ for some compact Hausdorff $K$, all six notions from Definition \ref{def-six-notions} are equivalent to attaining its maximum at a point $x\in K'$. In the vector-valued scenario, this condition implies being a ccs-Daugavet point, but being a $\Delta$ point does not imply this condition in general. 
%\item In $L_1$ preduals, all six notions are equivalent.
%\item In $L_1(\mu)$ spaces ($\mu$ any positive measure), super-Daugavet, Daugavet, super-$\Delta$ and $\Delta$ are all equivalent to not being supported in atoms, and in the real case they are also equivalent to being ccs-$\Delta$. But in general these notions are not equivalent to ccs-Daugavet in these spaces.
%\end{itemize}

In \cite[Question 7.11]{MPZpre} the authors posed an open question to characterize all these notions in other Banach spaces, such as uniform algebras and spaces of vector-valued functions. The main objective of this paper is to tackle this question for $C_0(L, X)$, $A(K, X)$ (in particular uniform algebras), $L_1(\mu, X)$, and $L_{\infty}(\mu, X)$ spaces.

The paper is structured as follows. In Section \ref{section-preliminaries}, we recall the necessary notations and preliminary results. In Section \ref{section:stability}, we recall well-known results (\cite{AHLP,HLPV,HPV21,MPZpre}) on the stability of $\nabla$ points, Daugavet points, $\Delta$ points, and their variants under $\oplus_\infty$- and $\oplus_1$-sums of Banach spaces, and we improve and extend some of those results. Finally, Section \ref{section-results} is devoted to studying these diametral notions in several classes of Banach spaces. In Subsection \ref{subsection:C0LX}, we provide characterizations of those points in $C_0(L,X)$ spaces (Theorem \ref{theo-charac-Daug-C0LX}). 
%We compare the results with those for $C_0(L)$ spaces (Theorem \ref{theo:implications-C0LX}). 
As a consequence, we also characterize several diametral diameter two properties in these spaces when $X$ is locally uniformly rotund (Corollary \ref{cor:c0lxequiv}). In Subsection \ref{subsection:AKX}, we show that all the results from Subsection \ref{subsection:C0LX} can be analogously stated for $A(K,X)$ spaces (see Section \ref{section-preliminaries} for the definition). 
%To do so, the Urysohn-type lemma from \cite{CGK13} and a geometrical bound lemma are needed. 
Hence, we characterize all 7 notions in uniform algebras $A(K)$ (Corollary \ref{cor:charac-unif-algebras}), which improves \cite[Corollary 5.5]{LT22} and fully solves \cite[Question 7.11]{MPZpre} for these spaces. Another consequence is an improvement of \cite[Proposition 4.15]{LT24} (Corollary \ref{cor:akxequiv}). In Subsection \ref{subsection:Linfty(mu,X)}, we characterize $\nabla$, $\Delta$, and Daugavet points in $L_\infty(\mu, X)$ spaces and provide conditions for the other notions. In particular, an easy description of all 7 types of points in $L_\infty(\mu)$ spaces is derived (see Theorem \ref{theo:charac-Daug-points-Linfty} and Example \ref{example:ell_infty}). Finally, in Subsection \ref{subsection:L1(mu,X)}, we characterize $\nabla$, $\Delta$, and Daugavet points in $L_1(\mu, X)$ spaces and provide several conditions for the other notions. We also improve \cite[Proposition 4.9]{MPZpre} and characterize ccs-$\Delta$ points in $L_1(\mu)$ spaces in the complex case, which was implicitly left open in \cite[Section 4]{MPZpre} (Theorem \ref{theorem:ccs-delta-L1} and Corollary \ref{cor:notions-L1(mu)}).

\section{Preliminaries}\label{section-preliminaries}

% \color{blue}Preliminaries: Properly introduce all the spaces we talk about (I think we don't need to mention what is known for them in this section, we can mention that in the results sections), all the equivalences of notions that we use, comments on complex nabla points, summary and extension of stability results for $\oplus_\infty$ and $\oplus_1$ sums and their consequences...\color{black}

%\section{\texorpdfstring{$\boldsymbol{C_0(L)}$}{C0(L)} spaces}

\subsection{The spaces \texorpdfstring{$\boldsymbol{C_0(L,X)}$}{C0(L,X)}, \texorpdfstring{$\boldsymbol{A(K,X)}$}{A(K,X)}, and \texorpdfstring{$\boldsymbol{L_p(\mu, X)}$}{Lp(mu,X)}} 
Let $L$ be a locally compact Hausdorff topological space. The space $C_0(L, X)$ is the set of $X$-valued continuous functions on $L$ that vanish at infinity, endowed with the supremum norm. That is, for each $f \in C_0(L, X)$ and $\epsilon > 0$, there exists a compact $K \subset L$ such that $\|f(t)\|_X < \epsilon$ for all $t \in L \setminus K$. If $X = \mathbb{R}$ or $\mathbb{C}$, then we obtain the usual $C_0(L)$ space. Also, if $L=K$ is compact, we have the usual $C(K, X)$ and $C(K)$ spaces, respectively. %\textcolor{red}{Should we add a reference for background?}

%\textcolor{red}{Should we comment that in the real case $A(K)=C(K)$? Reference?}

For a compact Hausdorff topological space $K$, a \textit{uniform algebra} $A(K)$ is a closed subspace of $C(K)$ that contains all constant functions and separates points. Here ``separating points" means that for every $s \neq t\in K$, there exists $f \in A(K)$ such that $f(s) \neq f(t)$. Examples of uniform algebras include the disk algebra $A(\mathbb{D})$ and the space $A_{w^*u}(B_{X^*})$ of weak* uniformly continuous functions that are holomorphic on the interior of $B_{X^*}$. In the case of uniform algebras, we only need to consider the range space to be $\mathbb{C}$ because the real uniform algebra is $C(K)$ by the Stone-Weierstrass Theorem (see \cite[pg 122]{Rudin91}).

To study the geometrical properties of uniform algebras and their related spaces, certain points in $K$ play an important role. A point $t_0 \in K$ is said to be a \textit{strong boundary point} if for every open neighborhood $U$ of $t_0$, there exists $f_U \in S_{A(K)}$ such that $f_U(t_0) = 1$ and $\sup_{t \in K \setminus U}|f_U(t)| < 1$.  A boundary $S$ for $A(K)$ is the set of elements in $K$ such that for every $f \in A(K)$, there exists an element $t \in S$ such that $|f(t)| = \|f\|_{\infty}$. The smallest closed boundary for $A(K)$ is called the {\it Shilov boundary}, denoted by $\Gamma$. It is known \cite{Dales, Leibowitz} that  the set of strong boundary points coincides with the Choquet boundary $\Gamma_0$ for $A(K)$, that is, the set of all extreme points on $K_{A(K)} = \{\lambda \in A(K)^* : \|\lambda\| = \lambda(1_{A(K)}) = 1\}$, where $1_{A(K)}$ is the unit of $A(K)$, and that $\Gamma_0$ is dense in $\Gamma$.

% In the literature, studying geometrical properties on $C(K)$ space often relies on using the Urysohn lemma. However, this lemma is inadequate to study the geometrical properties of uniform algebras $A(K)$ and their related spaces because the functions generated by the lemma may not be in $A(K)$ \cite{CGK13}. Hence, we will use the Urysohn-type functions whose existence is guaranteed by the following lemma.  

% \begin{Lemma}[{\cite[Lemma 2.5]{CGK13}}]\label{lemma:urysohn-type}
% Let $K$ be a compact Hausdorff topological space, $A(K)$ be a uniform algebra over $K$, and let $\Gamma_0$ be its Choquet boundary. Then, for every open set $U\subset K$ with $U\cap \Gamma_0\neq \emptyset$ and $0<\varepsilon<1$, there exists $f\in A(K)$ and $t_0\in U\cap \Gamma_0$ such that $f(t_0)=\|f\|_\infty=1$, $|f(t)|<\varepsilon$ for every $t\in K\setminus U$, and for all $t\in U$, we have $|f(t)|+(1-\varepsilon)|1-f(t)|\leq 1$.
% \end{Lemma}

Now, we recall the definition of a class of function spaces associated with uniform algebras. The space $A(K, X)$ is a closed subspace of $C(K, X)$ that satisfies the following properties:
\begin{enumerate}
\item The base algebra $A:= \{x^* \circ f : x^* \in X^* \,\,\, \text{and} \,\,\, f \in A(K,X)\}$ is a uniform algebra over $K$.
\item $A \otimes X \subset A(K,X)$.
\item For every $f \in A$ and $g \in A(K, X)$, $f\cdot g \in A(K, X)$.
\end{enumerate}
This space is endowed with the supremum norm $\|f\| = \sup_{t \in K}\|f(t)\|_X$. If $X = \mathbb{R}$ or $\mathbb{C}$, then we have a uniform algebra $A(K)$. For more information on this vector-valued function space, we refer to \cite{LT22,LT24}. %\textcolor{red}{Should we add references for background? $A(K,X)$ spaces are used for instance in \cite{LT22,LT24}}.
To prove certain results in $A(K, X)$, we will only consider the space of functions in $A(K)$ restricted to its Shilov boundary $\Gamma$, and we denote this by $A(\Gamma)$. We adopt the analogous notation for the space $A(K, X)$. It is well-known that $A(\Gamma)$ is isometrically isomorphic to $A(K)$ \cite[Theorem 4.1.6]{Leibowitz}.
%\textcolor{red}{Maybe we should add that $A(\Gamma)$ is isometric to $A(K)$.}

Let $(\Omega, \Sigma, \mu)$ be a  measure space, and let $X$ be a Banach space. Given a measurable set $A$ in $\Sigma$, $\chi_A$ is the {\it characteristic function} on $A$ which has the value 1 on $A$ and 0 on $\Omega\setminus A$.  A function $f:\Omega\to X$ is called {\it $\mu$-simple} if it is of the form
\[ f (\omega)= \sum_1^n x_i \chi_{A_i}(\omega)   \ \ \ (\omega\in \Omega),\] where  $x_1, \dots, x_n\in X$. 
Recall that a function $f:\Omega\to X$ is said to be a strongly $\mu$-measurable if it is a pointwise limit of a sequence of $\mu$-simple functions. Two strongly measurable functions $f, g$ are said to be {\it equivalent} if they are equal $\mu$-a.e.

For $1 \leq p < \infty$, $L_{p}(\mu, X)$ is the Banach space of equivalent classes of  $p$-Bochner integrable functions, that is,   strongly $\mu$-measurable functions satisfying
\[\|f\|_p = \left(\int_{\Omega} \|f(t)\|^p_X d\mu\right)^{1/p} < \infty.\] 

In the case of $p = \infty$, the space $L_{\infty}(\mu, X)$ is defined to be the Banach space of all  equivalent classes of essentially bounded strongly $\mu$-measurable functions $f$, that is, $f$ satisfies the condition
\[ \mu(\{ \omega : \|f(\omega)\|_X>r\})=0\] for some $r>0$.
In this case, the essential supremum norm   $\|f\|_\infty$ is defined to be the infimum of all $r>0$ satisfying $\mu(\{ \omega : \|f(\omega)\|_X>r\})=0$. 

If $X=\bbr$ or $\bbc$, then we get the usual $L_p(\mu)$ spaces. To characterize the notions in $L_p(\mu, X)$ with full generality, we do not assume the $\sigma$-finiteness and nonatomicity of $\mu$ here. Also, we only need to consider atoms of finite measure, since every strongly $\mu$-measurable function has value $0$ a.e. on atoms of infinite measure. We say that two atoms $A$ and $B$ are identified if their characteristic functions $\chi_A$ and $\chi_B$ are equivalent. 

As a final remark, there is an alternative space $L_\infty(\mu, X)$ in the literature, where the functions are assumed to be strongly measurable instead of strongly $\mu$-measurable (that is, the supports of characteristic functions may not have finite measure). Although this space presents some differences with the one we will use, we remark that all our results in Subsection \ref{subsection:Linfty(mu,X)} are also true for that space without any changes.

%\textcolor{red}{Should we add a reference for background, such as \cite{Diest}?}

%\textcolor{red}{Do we need to add any more stuff about the spaces?}

%\textcolor{red}{I removed the condition that $\mu$ is positive, since it is a usual convention to omit it. I found that there are two different $L_\infty(\mu, X)$ spaces. 
%A function $f:\Omega\to X$ is called a {\it simple function} if its range is a finite set. Note that a simple function is a  $\mu$-simple function if its support has a finite measure. 
%A function $f:\Omega\to X$ is called a {\it strongly measurable} if it is a pointwise limit of simple functions. 
%Antother space $\tilde{L}_\infty(\mu)$ is the space of equivalence classes of essentially bounded strongly measurable functions. For example, if $\mu$ is not $\sigma$-fintie, $1\otimes x$ $(x\in X)$ is not strongly $\mu$-measurable, but is is strongly measurable.}

\subsection{Further preliminary results}
We conclude this section with remarks on certain classes of points in the unit sphere that will be relevant for later discussions.

Given a Banach space $X$, a point $x\in S_X$, and some $\varepsilon>0$, denote $\Delta_\varepsilon(x):=\{y\in B_X:\, \|x-y\|\geq 2-\varepsilon\}$. The following well-known characterizations will be used throughout the text without a reference.

\begin{lemma}[{\cite[Lemmas 2.1 and 2.2]{AHLP} real case, \cite[Theorems 2.2 and 2.4]{LT24} complex case, and \cite[Proposition 3.4]{MPZpre}}]
Let $X$ be a Banach space, and let $x\in S_X$. Then,
\begin{itemize}
\item $x$ is $\Delta$ if and only if $x\in \overline{\operatorname{co}(\Delta_\varepsilon(x))}$.
\item $x$ is Daugavet if and only if $B_X= \overline{\operatorname{co}(\Delta_\varepsilon(x))}$ for every $\epsilon>0$.
\item $x$ is super-$\Delta$ if and only if there is a net $\{x_\lambda\}_{\lambda\in \Lambda}\subset B_X$ such that $x_\lambda \stackrel{w}{\rightarrow}x$ but $\|x-x_\lambda\|\rightarrow 2$.
\item $x$ is super-Daugavet if and only if for every $z\in B_X$, there is a net $\{x_\lambda\}_{\lambda\in \Lambda}\subset B_X$ such that $x_\lambda \stackrel{w}{\rightarrow}z$ but $\|x-x_\lambda\|\rightarrow 2$.
\end{itemize}
\end{lemma}

Recall that $x\in S_X$ is \textit{strongly exposed} if there is a functional $x^*\in S_{X^*}$ such that for every $\varepsilon>0$, there is some $\delta>0$ such that the slice $S(x^*, \delta)$ contains $x$ and has diameter less than $\varepsilon$. In this case we say that $x$ is strongly exposed by $x^*$, or that $x^*$ strongly exposes $x$. On the other hand, $x\in S_X$ is \textit{denting} if for every $\varepsilon>0$, there are a functional $x^*\in S_{X^*}$ and some $\delta>0$ such that the slice $S(x^*, \delta)$ contains $x$ and has diameter less than $\varepsilon$. Note that strongly exposed points are always denting.

% \subsection{Some things that have to be defined in the paper, here or anywhere else}
% We can copy and paste many of these things from the backup file. Here's a list of things that we use and hence should be introduced.

% \begin{itemize}
% \item Strongly exposed point and denting point.

% \color{red}
% Recall that $x\in S_X$ is strongly exposed if there is some functional $x^*\in S_{X^*}$ with $x^*(x)=1$ and such that whenever $\re(x^*(x_n))\rightarrow \re(x^*(x))=1$ for a sequence of points $\{x_n\}_{n\in\bbn}\subset B_X$, we must necessarily have $x_n\rightarrow x$. 

% Equivalently, $x\in S_X$ is strongly exposed if there is a functional $x^*\in S_{X^*}$ such that for every $\varepsilon>0$, there is some $\delta>0$ such that the slice $S(x^*, \delta)$ contains $x$ and has diameter less than $\varepsilon$. In this case we say that $x$ is strongly exposed by $x^*$, or that $x^*$ strongly exposes $x$. 

% On the other hand, $x\in S_X$ is denting if for every $\varepsilon>0$, there are a functional $x^*\in S_{X^*}$ and some $\delta>0$ such that the slice $S(x^*, \delta)$ contains $x$ and has diameter less than $\varepsilon$. Note that strongly exposed points are trivially denting.
% \color{black}

% \item \textcolor{red}{Anything else?}
% \end{itemize}

% \subsection{Some things we use}
% \textcolor{red}{Define $\Delta_\varepsilon(x)$}

% \color{red}

% \color{black}

% Also:

In \cite[Theorem 2.3, Proposition 2.6]{HLPV} it was shown that $\nabla$ points $x\in S_X$ in a \textit{real} Banach space must be either Daugavet or strongly exposed, and must be at distance $2$ from every denting point $y\neq x$. It is easy to check that those results also hold in the complex case, and from that we get that in a complex Banach space, $\nabla$ and Daugavet points actually always coincide.

\begin{proposition}\label{prop-nablas-complex}
In a complex Banach space $X$, $x\in S_X$ is a $\nabla$ point if and only if it is a Daugavet point.
\end{proposition}

\begin{proof}
Let $X$ be a complex Banach space, and suppose that $x\in S_X$ is a $\nabla$ point that is not Daugavet. By \cite[Theorem 2.3]{HLPV}, $x$ is a strongly exposed point. But then, $\theta x$ is also strongly exposed for every $\theta\in S_{\bbc}$. In particular, $x$ is at distance less than $2$ from a denting point, which contradicts \cite[Proposition 2.6]{HLPV}.
\end{proof}

\section{Stability results}\label{section:stability}

In this section, we gather some well-known (see \cite{AHLP,HLPV,HPV21,MPZpre}) and some new additional stability results on the diametral points under $\oplus_1$- and $\oplus_\infty$-sums that will be used throughout this paper. Here the term ``downward stability" stands for the inheritance from those sums to their component spaces. On the other hand, the term ``upward stability" means the inheritance from component spaces to their sums. It is worth noting that the stability results from \cite{AHLP,HLPV,HPV21} were only stated and proved in the real case, but a quick glance at the proofs shows that all of them can be easily adapted to the complex case.

\subsection{\texorpdfstring{$\boldsymbol{\nabla}$}{Nabla} points}\label{subsection:stability-nabla}

The stability of $\nabla$ points under absolute sums was studied in detail in \cite{HLPV} for \textit{real} Banach spaces, although as we mentioned, those results also hold in the complex case. In the case of $\oplus_\infty$-sums, they can only be different from Daugavet points if both coordinates have norm 1.

\begin{lemma}[{\cite[Proposition 3.5]{HLPV}}]\label{nabla-infty-sums}
Let $X_1, X_2$ be  Banach spaces, let $x_1\in S_{X_1}$, $x_2\in S_{X_2}$, and $b\in [0,1]$. The following holds:
\begin{enumerate}
\item If $b<1$, then $(x_1, b x_2)\in S_{X_1\oplus_\infty X_2}$ is $\nabla$ if and only if $x_1$ and $(x_1, b x_2)$ are Daugavet.
\item If $b=1$, then $(x_1, x_2)\in S_{X_1\oplus_\infty X_2}$ is $\nabla$ if and only if either at least one of $x_1$ or $x_2$ is Daugavet, or $x_1$ and $x_2$ are both $\nabla$.
\end{enumerate}
\end{lemma}

As for the case of $\oplus_1$-sums of spaces, they can only differ from Daugavet points when one of the coordinates is $0$.

\begin{lemma}[{\cite[Propositions 3.3 and 3.4]{HLPV}}]\label{nabla-1-sums}
Let $X_1, X_2$ be Banach spaces, and for each $i\in\{1,2\}$ let $x_i\in S_{X_i}$ and $a_i\in (0,1]$ be such that $a_1+a_2=1$. The following holds.
\begin{enumerate}
\item $(x_1,0)\in S_{X_1\oplus_1 X_2}$ is $\nabla$ if and only if so is $x_1$.
\item $(a_1 x_1, a_2 x_2)\in S_{X_1\oplus_1 X_2}$ is $\nabla$ if and only if $x_1$, $x_2$, and $(a_1 x_1, a_2 x_2)$ are all Daugavet.
\end{enumerate}
\end{lemma}

Note that these characterizations can be naturally extended to any finite sum of spaces.

\begin{remark}
Recall that in the complex case $\nabla$ and Daugavet points coincide (see Proposition \ref{prop-nablas-complex}). It is worth noting that in this scenario, the characterizations above coincide with the respective stability results for Daugavet points from the literature, as expected.
\end{remark}

\subsection{Daugavet and \texorpdfstring{$\boldsymbol{\Delta}$}{Delta} points}\label{subsection:stability-Daugavet-Delta}

The stability of Daugavet points under absolute sums has been characterized in \cite{HPV21}. In the case of $\oplus_\infty$-sums, the following has been shown for the sum of two spaces and can be easily generalized to any finite sum. 

\begin{lemma}[{\cite[Proposition 2.4 and Theorem 3.2]{HPV21}}]\label{daugavet-infty-sums}
Let $\{X_k\}_{k=1}^n$ be a finite collection of Banach spaces. For each $1\leq k\leq n$, let $x_k\in S_{X_k}$ and $a_k\geq 0$ be such that $(a_1,\ldots,a_n)\in S_{\ell_\infty^n}$. Then $(a_1 x_1,\ldots,a_n x_n)\in S_{X_1\oplus_\infty\ldots\oplus_\infty X_n}$ is Daugavet if and only if there is some $k\in\{1,\ldots,n\}$ such that $a_k=1$ and $x_k$ is Daugavet.
\end{lemma}

For $\Delta$ points the situation is more complicated, as there are Banach spaces $X,Y$ and points $x\in S_X$ and $y\in S_Y$ such that $(x,y)$ is $\Delta$ in $S_{X\oplus_\infty Y}$, but none of $x$ or $y$ are $\Delta$ (see \cite[Section 4]{HPV21}). Nevertheless, a characterization of the $\Delta$ points in $X\oplus_\infty Y$ can be achieved by combining \cite[Propositions 4.4 and 4.5]{HPV21} with \cite[Claim inside the proof of Theorem 5.8]{AHLP} and studying the final missing case (where only one of the coordinates is $\Delta$, which can be tackled analogously to the last cited result). To extend the result to finite sums, we make a refinement of the computations done in those results. We include the details.

Recall (\cite[Definition 4.2]{HPV21}) that $x\in S_X$ is a \textit{$\Delta_p$ point} with $p>1$ if for every slice $S(B_X, x^*, \alpha)$ containing $x$ and every $\varepsilon>0$, there is $u\in S(B_X, x^*, p\alpha)$ such that $\|x-u\|\geq 2-\varepsilon$. Note that every $\Delta$ point is $\Delta_p$ for all $p>1$, and that every $\Delta_p$ point is $\Delta_r$ for all $1<p\leq r$.

\begin{lemma}\label{delta-infty-sums}
Let $n>1$ be an integer, let $X_1,\ldots,X_n$ be Banach spaces, and for each $1\leq k\leq n$, let $x_k\in S_{X_k}$. Let $(a_1,\ldots,a_n)\in \ell_\infty^n$ be such that $a_k\geq 0$ for all $1\leq k\leq n$. Then $(a_1 x_1,\ldots,a_n x_n)\in S_{X_1\oplus_\infty \ldots\oplus_\infty X_n}$ is $\Delta$ if and only if either of the following conditions holds:
\begin{enumerate}
\item There is $1\leq k\leq n$ with $a_k=1$ such that $x_k$ is $\Delta$.
\item If $A:=\{k\in\{1,\ldots, n\}:\, a_k=1\}$, then for every $\{p_k\}_{k\in A}$ with $p_k>1$ for each $k\in A$ and $\sum_{k\in A}\frac{1}{p_k}=1$, there is some $k\in A$ such that $x_k$ is $\Delta_{p_k}$.
\end{enumerate}
In particular, if $A$ only has one element, say $k$,  then $(a_1 x_1,\ldots,a_n x_n)$ is $\Delta$ if and only if $x_k$ is $\Delta$.
\end{lemma}

\begin{proof}
Write for simpliticy $A=\{k_1,\ldots,k_m\}$, $C:=\{1,\ldots,n\}$, and $D:=C\setminus A=\{k_{m+1},\ldots,k_{n}\}$ if $m<n$. Note that if $D\neq\emptyset$, then we can split $(a_1 x_1,\ldots,a_n x_n)$ as
$$((a_{k_1}x_{k_1},\ldots,a_{k_m}x_{k_m}),(a_{k_{m+1}}x_{k_{m+1}},\ldots,a_{k_{n}}x_{k_{n}}))\in \left[\bigoplus_{r=1}^m X_{k_r}\right]_\infty \oplus_\infty \left[\bigoplus_{r=m+1}^n X_{k_r}\right]_\infty.$$
Thus, if $(a_1 x_1,\ldots,a_n x_n)$ is $\Delta$, then so is $(a_{k_1}x_{k_1},\ldots,a_{k_m}x_{k_m})\in \left[\bigoplus_{r=1}^m X_{k_r}\right]_\infty$ by \cite[Theorem 4.1]{HLPV}. 

The converse is also true, and in fact we show more by arguing analogously to the Claim inside \cite[Theorem 5.8]{AHLP}. Indeed, let $\varepsilon>0$ be given, and without loss of generality assume that $1\in A$, and suppose that $x_1$ is $\Delta$ (the same argument can be applied to any other coordinate). Then by \cite[Lemma 4.1]{AHLP} we can find some $z_1,\ldots,z_l\in S_{X_1}$ such that
$$\left\| x_1 - \frac{1}{l}\sum_{j=1}^l z_j\right\| < \varepsilon\quad \text{and}\quad \|x_1-z_j\|\geq 2-\varepsilon\text{ for all }1\leq j\leq l.$$
Now, on the one hand, we have
$$\left\| (x_1,a_2 x_2,\ldots,a_n x_n) - \frac{1}{l}\sum_{j=1}^l (z_j,a_2 x_2,\ldots,a_n x_n)\right\| = \left\| x_1 - \frac{1}{l}\sum_{j=1}^l z_j\right\| < \varepsilon,$$
and on the other hand, for each $1\leq j\leq l$, we have
$$\|(x_1,a_2 x_2,\ldots,a_n x_n)-(z_j,a_2 x_2,\ldots,a_n x_n)\|=\|x_1-z_j\|\geq 2-\varepsilon.$$
This shows that $(a_1 x_1, \ldots, a_n x_n)$ is a $\Delta$ point in this case.

Therefore, to prove the rest of the result, it suffices only to consider (for the sake of notational simplicity) the case where $a_1=\ldots=a_n=1$, that is, $A=\{1,\ldots,n\}$, since any other case can be reduced to this one by the previous argument, we will assume this from now on. 

We have seen that \textup{(1)} implies that $(x_1,\ldots,x_n)$ is $\Delta$. Suppose now that $(x_1,\ldots,x_n)$ is $\Delta$ but \textup{(1)} does not hold, and we will see by contradiction that then \textup{(2)} holds. Indeed, assume that there are $\{p_k\}_{k\in A}$ with $p_k>1$ for all $k\in A$ and $\sum_{k\in A}\frac{1}{p_k}=1$ such that for all $k\in A$, $x_k$ is not $\Delta_{p_k}$. Then there is $\varepsilon>0$, and for each $k\in A$ there are $x_k^*\in S_{X_k^*}$ and $\alpha_k>0$ such that $x_k\in S(B_{X_k}, x_k^*, \alpha_k)$, but for all $u_k\in S(B_{X_k}, x_k^*, p_k \alpha_k)$, we have $\|x_k-u_k\|<2-\varepsilon$. For each $k\in A$, define 
$$\lambda_k:=\frac{\prod\limits_{j\in A\setminus\{k\}} p_j \alpha_j}{\sum\limits_{i\in A} \prod\limits_{j\in A\setminus\{i\}} p_j \alpha_j}.$$
Let $\alpha=\sum_{k\in A} \lambda_k \alpha_k$ and $f^*:=(\lambda_1 x_1^*,\ldots,\lambda_n x_n^*)\in S_{X_1^*\oplus_1\ldots \oplus_1 X_n^*}$, and note that
$$\re(f^*((x_1,\ldots,x_n)))=\lambda_1\re(x_1^*(x_1))+\ldots+\lambda_n\re(x_n^*(x_n))>\sum_{k=1}^n \lambda_k (1-\alpha_k)=1-\alpha.$$
Choose any $(u_1,\ldots,u_n)\in S(B_{X_1\oplus_\infty\ldots\oplus_\infty X_n}, f^*, \alpha)$. We have
$$\re(f^*((u_1,\ldots,u_n)))=\lambda_1\re(x_1^*(u_1))+\ldots+\lambda_n\re(x_n^*(u_n))>1-\alpha.$$
Fix any $1\leq k\leq n$. Since
$$1-\alpha < \lambda_k \re(x_k^*(u_k)) + \sum_{j\in A\setminus\{k\}} \lambda_j \re(x_j^*(u_j))\leq \lambda_k\re(x_k^*(u_k)) + \sum_{j\in A\setminus\{k\}} \lambda_j,$$
we get
\begin{align*}
\re(x_k^*(u_k))&>\frac{\lambda_k-\alpha}{\lambda_k}=1-\left( \alpha_k + \frac{\sum_{j\in A\setminus\{k\}} \alpha_j\lambda_j}{\lambda_k} \right) = 1-\left( \frac{p_k\alpha_k}{p_k} + \sum_{j\in A\setminus\{k\}} \alpha_j\frac{p_k \alpha_k}{p_j \alpha_j} \right) \\
&= 1-p_k \alpha_k.
\end{align*}
Thus, by the assumption, $\|x_k-u_k\|<2-\varepsilon$ for all $k\in A$, but this contradicts the fact that $(x_1,\ldots,x_n)$ is $\Delta$ and proves the claim.

Finally, it remains to show that if \textup{(2)} holds but \textup{(1)} does not, then $(x_1,\ldots,x_n)$ is $\Delta$. We will show this in 2 steps. 

\textit{Step 1:} First we will see that, under these assumptions, there is some $\emptyset\neq B\subset A$ and some $\{p_k\}_{k\in B}$ with $p_k>1$ for all $k\in B$ and $\sum_{k\in B}\frac{1}{p_k}=1$ such that for all $k\in B$ $x_k$ is a $\Delta_{p_k}$ point. For each $k\in A$, let $B_k:=\{j\in [1,+\infty):\, x_k\text{ is }\Delta_j\text{ in }X_k\}$, and write $b_k:=\inf\{j:\, j\in B_k\}$. 

First, we show that if for some $k\in A$ we have $B_k\neq\emptyset$, then $b_k\in B_k$. Indeed, fix $\varepsilon>0$, $x_k^*\in S_{X_k^*}$, and $\alpha_k>0$ such that $x_k\in S(B_{X_k}, x_k^*, \alpha_k)$. Let $\gamma_k>0$ be such that $\re(x_k^*(x_k))>1-(\alpha_k - \gamma_k)$, and let $j_k=\frac{b_k \alpha_k}{\alpha_k-\gamma_k}>b_k$, so $j_k\in B_k$. Thus, there is $u_k\in S(B_{X_k}, x_k^*, j_k (\alpha_k - \gamma_k))=S(B_{X_k}, x_k^*,b_k \alpha_k)$ such that $\|u_k-x_k\|\geq 2-\varepsilon$. But this means that $x$ is $\Delta_{b_k}$, that is, $b_k\in B_k$. 

Now, notice that in order to satisfy \textup{(2)}, at least one of the sets $B_k$ ($1\leq k\leq n$) must be non-empty, and if \textup{(1)} does not hold, in particular, at least 2 of such sets must be nonempty. Let $B:=\{k\in A:\, B_k\neq\emptyset\}$. We will see that $\sum_{k\in B}\frac{1}{b_k}\geq 1$. Indeed, otherwise there would exist $\delta>0$, and if $B\neq A$ there would also exist some $\{c_r\}_{r\in A\setminus B}$ with $c_r>1$ for each $r\in A\setminus B$, such that
$$\sum_{r\in B} \frac{1}{b_r-\delta} + \sum_{s\in A\setminus B}\frac{1}{c_s}=1.$$
But by assumption, this would imply that there is some $k\in B$ such that $x_k$ is $\Delta_{b_k-\delta}$. This contradicts the definition of $b_k$ and finishes this claim.

\textit{Step 2:} Finally, we will show that if such $B$ and $\{p_k\}_{k\in B}$ exist, then $(x_1,\ldots,x_n)$ is $\Delta$. Once more, to simplify notation, by the first part of the proof, it suffices to show this for the case where $B=A$. Fix $\varepsilon>0$, $f^*=(x_1^*, .\ldots, x_n^*)\in S_{X_1^*\oplus_1\ldots\oplus_1 X_n^*}$, and $\alpha>0$ such that $(x_1,\ldots,x_n)\in S(B_{X_1\oplus_\infty \ldots\oplus_\infty X_n}, f^*, \alpha)$. Then,
$$(\re(x_1^*(x_1))-\|x_1^*\|)+\ldots+(\re(x_n^*(x_n))-\|x_n^*\|)>-\frac{\alpha}{p_1}-\ldots-\frac{\alpha}{p_n}.$$
Note that there is some $k\in A$ such that $\re(x_k^*(x_k))>\|x_k^*\|-\frac{\alpha}{p_k}$, since otherwise we would have that $\re(f^*((x_1,\ldots,x_n)))\leq 1-\alpha$, which is a contradiction. Let such $k$ be given. Then there is some $u_k\in S\left(B_{X_k}, \frac{x_k^*}{\|x_k^*\|_{X_k^*}}, \frac{\alpha}{\|x_k^*\|_{X_k^*}}\right)$ such that $\|x_k-u_k\|\geq 2-\varepsilon$. Fix $\eta>0$ small enough so that $\re(x_k^*(u_k))>\|x_k^*\|-\alpha + \eta$, and for  $j\in A\setminus\{k\}$, find any $u_j\in S_{X_j}$ such that $\re(x_j^*(u_j))>\|x_j^*\|-\frac{\eta}{n}$. Notice that
$$\re(f^*((u_1,\ldots,u_n)))>1-\alpha,$$
and clearly
$$\|(u_1,\ldots,u_n)-(x_1,\ldots,x_n)\|\geq \|u_k-x_k\|\geq 2-\varepsilon.$$
This shows that $(x_1,\ldots,x_n)$ is $\Delta$, and the proof is finished.
\end{proof}

In the case of $\oplus_1$-sums, $\Delta$ and Daugavet points were also characterized in \cite{AHLP,HPV21} for the sum of two spaces, and the result can be easily extended once more to any finite sum of spaces.

\begin{lemma}[{\cite[Claim inside Theorem 5.8]{AHLP} and \cite[Theorems 2.2, 3.1, and 4.1]{HPV21}}]\label{delta-daug-1-sums}
Let $\{X_k\}_{k=1}^n$ be a finite collection of Banach spaces. For each $1\leq k\leq n$, let $x_k\in S_{X_k}$ and $a_k\geq 0$ be such that $(a_1,\ldots,a_n)\in S_{\ell_1^n}$. Then $(a_1 x_1,\ldots,a_n x_n)\in S_{X_1\oplus_\infty\ldots\oplus_\infty X_n}$ is $\Delta$ (respectively Daugavet) if and only if for every $k\in \{1,\ldots,n\}$ such that $a_k\neq 0$, the point $x_k\in S_{X_k}$ is $\Delta$ (respectively Daugavet).
\end{lemma}

In fact, we show now that this result can be extended to countable $\ell_1$-sums of Banach spaces, as this fact will be needed in a later section.

\begin{lemma}\label{delta-daug-ell1-countable-sums}
Let $\{X_k\}_{k=1}^\infty$ be a countable collection of Banach spaces, and let $Z:=\left[ \bigoplus_{k=1}^\infty X_k\right]_{\ell_1}$. For each $k\in \bbn$, let $x_k\in S_{X_k}$ and $a_k\geq 0$ be such that $(a_1,a_2,\ldots)\in S_{\ell_1}$. Then $(a_1 x_1, a_2 x_2, \ldots)\in S_{Z}$ is $\Delta$ (respectively Daugavet) if and only if for every $k\in \bbn$ such that $a_k\neq 0$, the point $x_k\in S_{X_k}$ is $\Delta$ (respectively Daugavet).
\end{lemma}

Note that, as usual, by the properties of absolute sums, it suffices to consider the case where $a_n\geq 0$ for all $n\in\bbn$.

\begin{proof}
Assume first that $(a_1 x_1, a_2 x_2, \ldots)$ is $\Delta$ (respectively Daugavet), and let $n\in\bbn$ be such that $a_n>0$. For simplicity, assume that $n=1$ (the same argument can be used for any other value of $n$). Then $(a_1 x_1, (a_2 x_2, a_3 x_3, \ldots))\in X_1\oplus_1 \left[ \bigoplus_{k=2}^\infty X_k\right]_{\ell_1}$ is also $\Delta$ (respectively Daugavet), and by Lemma \ref{delta-daug-1-sums}, so is $x_1\in X_1$.

For the converse, note that the case where $S=\{n\in \bbn:\, a_n\neq 0\}$ is a finite set is a consequence of applying Lemma \ref{delta-daug-1-sums}, so we assume now that $S$ is infinite. Moreover, the case where $\bbn\setminus S\neq \emptyset$ can be reduced to the case where $S=\bbn$ (as $x\in S_X$ is $\Delta$ or Daugavet if and only if so is $(x,0)\in X\oplus_1 Y$ for Banach spaces $X$ and $Y$). So assume without loss of generality that for all $n\in\bbn$, $a_n> 0$. Fix $\varepsilon>0$. Then there is $k\in\bbn$ such that 
$$\left\|(a_1 x_1, a_2 x_2, \ldots) - \frac{1}{\sum_{n=1}^k |a_n|}(a_1 x_1, \ldots, a_k x_k, 0, 0, \ldots) \right\|< \frac{\varepsilon}{2}.$$
For simplicity, we write  
$$x:=(a_1 x_1, a_2 x_2, \ldots)\quad \text{and}\quad y=(y_1,y_2,\ldots):=\frac{1}{\sum_{n=1}^k |a_n|}(a_1 x_1, \ldots, a_k x_k, 0, 0, \ldots).$$
We distinguish 2 cases now.

\textit{Case 1:} Suppose that for all $n\in\bbn$, $x_n\in S_{X_n}$ is $\Delta$. Then $y\in S_{Z}$ is also $\Delta$. By \cite[Lemma 4.1]{AHLP}, there exist $m\in\bbn$ and $y_1,\ldots,y_m\in S_{Z}$ such that
$$\left\| y - \frac{1}{m}\sum_{n=1}^m y_n\right\|<\frac{\varepsilon}{2},\quad \text{and}\quad \|y-y_n\|\geq 2-\frac{\varepsilon}{2} \text{ for }n=1,\ldots,k.$$
But then triangle inequality yields that
$$\left\| x - \frac{1}{m}\sum_{n=1}^m y_n\right\|<\varepsilon,\quad \text{and}\quad \|x-y_n\|\geq 2-\varepsilon \text{ for }n=1,\ldots,k.$$
This shows that $x$ is $\Delta$.

\textit{Case 2:} Suppose that for all $n\in\bbn$, $x_n\in S_{X_n}$ is Daugavet. Then $y\in S_{Z}$ is also Daugavet. Fix $z\in S_{Z}$. By \cite[Lemma 4.1]{AHLP}, there exist $m\in\bbn$ and $y_1,\ldots,y_m\in S_{Z}$ such that
$$\left\| z - \frac{1}{m}\sum_{n=1}^m y_n\right\|<\frac{\varepsilon}{2},\quad \text{and}\quad \|y-y_n\|\geq 2-\frac{\varepsilon}{2} \text{ for }n=1,\ldots,k.$$
But then the triangle inequality yields that
$$\|x-y_n\|\geq 2-\varepsilon \text{ for }n=1,\ldots,k.$$
Since this can be done for any arbitrary $z\in S_{Z}$, this shows that $x$ is Daugavet.
\end{proof}

\subsection{Super-Daugavet and super-\texorpdfstring{$\boldsymbol{\Delta}$}{Delta} points}\label{subsection:stability-super-Daugavet-Delta}

The stability of super-Daugavet and super-$\Delta$ points was first studied in \cite{MPZpre}. In the case of $\oplus_\infty$-sums, we have the following partial characterization.

\begin{lemma}\label{super-infty-sums}
Let $X,Y$ be Banach spaces, let $x\in S_X$, $y\in S_Y$, and $b\in [0,1]$. The following statements hold:
\begin{enumerate}
\item If $x$ is super-$\Delta$ (respectively super-Daugavet), then so is $(x,by)\in S_{X\oplus_\infty Y}$.
\item If $(x,by)\in S_{X\oplus_\infty Y}$ is super-$\Delta$ (respectively super-Daugavet) with $b<1$, then so is $x$.
\end{enumerate}
\end{lemma}

\begin{proof}
\textup{(1)} was shown in \cite[Propositions 3.26 and 3.28 and Remark 3.29]{MPZpre}, so let us see \textup{(2)}. 

Let $(x,by)\in S_{X\oplus_\infty Y}$ be a super-Daugavet point, and let $\delta>0$ be such that $b<1-\delta$. Fix any $(z,t)\in B_{X\oplus_\infty Y}$. Then there is some net $\{(x_\lambda, y_\lambda)\}_{\lambda\in \Lambda}\subset B_{X\oplus_\infty Y}$ such that $(x_\lambda, y_\lambda)\stackrel{w}{\rightarrow}(z,t)$ and $\|(x,by)- (x_\lambda, y_\lambda)\|\rightarrow 2$. But note that then $\{x_\lambda\}_{\lambda\in \Lambda}\stackrel{w}{\rightarrow} z$, and since $\|by-y_\lambda\|\leq 2-\delta$ for all $\lambda\in \Lambda$, we must have $\|x-x_\lambda\|\rightarrow 2$. This shows that $x\in S_X$ is a super-Daugavet point. For super-$\Delta$ points, it suffices to repeat the same argument, just taking $(z,t)=(x,by)$.
\end{proof}

In the case of $\oplus_1$-sums, the following upwards stability result was shown in \cite{MPZpre}.

\begin{lemma}[{\cite[Propositions 3.26 and 3.28]{MPZpre}}]\label{super-1-sums}
Let $X_1, X_2$ be Banach spaces, and for each $i\in\{1,2\}$ let $x_i\in S_{X_i}$ and $a_i\in (0,1]$ be such that $a_1+a_2=1$. The following holds.
\begin{enumerate}
\item If $x_1$ is super-$\Delta$ (respectively, super-Daugavet), then so is $(x_1,0)\in S_{X_1\oplus_1 X_2}$.
\item If $x_1$ and $x_2$ are both super-$\Delta$ (respectively super-Daugavet), then so is $(a_1 x_1, a_2 x_2)\in S_{X_1\oplus_1 X_2}$.
\end{enumerate}
\end{lemma}

\subsection{ccs-Daugavet and ccs-\texorpdfstring{$\boldsymbol{\Delta}$}{Delta} points}\label{subsection:stability-ccs-Daugavet-Delta}

The stability of ccs-Daugavet points was first studied in \cite{MPZpre}. In the case of $\oplus_\infty$-sums, the following partial characterization can be achieved.

\begin{lemma}\label{ccs-Daugavet-infty-sums}
Let $X,Y$ be Banach spaces, let $x\in S_X$, $y\in S_Y$, and $b\in [0,1]$. The following holds:
\begin{enumerate}
\item If $x$ is ccs-Daugavet, then so is $(x,by)\in S_{X\oplus_\infty Y}$.
\item If $(x,by)\in S_{X\oplus_\infty Y}$ is ccs-Daugavet for some $b<1$, then so is $x$.
\end{enumerate}
\end{lemma}

\begin{proof}
\textup{(1)} was shown in \cite[Theorem 3.32]{MPZpre}, so let us see \textup{(2)}. 

Let $(x,by)\in S_{X\oplus_\infty Y}$ be a ccs-Daugavet point with $\|x\|=1>|b|$. Let $\varepsilon>0$. Let $C=\sum_{i=1}^n \lambda_i S_i$ be a convex combination of slices of $B_X$ where for each $1\leq i\leq n$, $S_i:=S(B_x, x_i^*, \delta_i)$ for some $\delta_i>0$ and some $x_i^*\in S_{X^*}$. Fix $0<\delta<\varepsilon$ such that $|b|<1-\delta$. For each $1\leq i\leq n$, set $f_i^*:=(x_i^*, 0^*)\in S_{X^*\oplus_1 Y^*}$, $\widetilde{S_i}:=S(B_{X\oplus_\infty Y}, f_i^*, \delta_i)$, and $\widetilde{C}=\sum_{i=1}^n \lambda_i\widetilde{S_i}$ is a convex combination of slices of $B_{X\oplus_\infty Y}$. Thus, there is $(z,t)\in \widetilde{C}$ with $\|(x,by)-(z,t)\|=\|x-z\|>2-\delta$. For each $1\leq i\leq n$, there is $(z_i, t_i)\in \widetilde{S_i}$ such that
$$z=\sum_{i=1}^n \lambda_i z_i,\quad \text{and}\quad t=\sum_{i=1}^n \lambda_i t_i.$$
Note that $1-\delta_i < \re(f_i^*(z_i, t_i))=\re(x_i^*(z_i))$, so $z_i\in S_i$, and thus, $z\in C$.
\end{proof}

As for ccs-$\Delta$ points, an easy adaptation of the previous proof provides the following downwards stability result.

\begin{lemma}\label{ccs-Delta-infty-sums}
Let $X,Y$ be Banach spaces, let $x\in S_X$, $y\in S_Y$, and $b\in [0,1]$. If $(x,by)\in S_{X\oplus_\infty Y}$ is ccs-$\Delta$ for some $b<1$, then so is $x$.
\end{lemma}

\begin{proof}
It suffices to repeat the proof of Lemma \ref{ccs-Daugavet-infty-sums}, just noting that if $x\in C$, then $(x,by)\in \tilde{C}$ (this is because for every $(z,by)$ and every $1\leq i\leq n$, we have $f_i^*(z,by)=x_i^*(z)$). This concludes the proof.
\end{proof}

% \textcolor{red}{Remark: The reason why we don't get upwards even when $b=0$ is that when you write down the \textit{straightforward argument}, as the functionals are arbitrary, one cannot deduce that if $(x,0)\in \widetilde{C}$ then $x\in C$.}

The case of $\oplus_1$-sums is more delicate, however. In \cite[Proposition 3.30]{MPZpre} it was shown that if either $X$ or $Y$ has a strongly exposed point, then $X\oplus_1 Y$ fails to have ccs-Daugavet points. In fact, it is enough to assume that one of the spaces has a denting point.

\begin{proposition}\label{prop:denting-no-ccs-1-sums}
Let $X$ and $Y$ be two Banach spaces such that $X$ contains a denting point $x\in S_X$. Then $X\oplus_1 Y$ has convex combinations of slices around $0$ of arbitraryly small diameter. In particular, $X\oplus_1 Y$ fails to contain ccs-Daugavet points.
\end{proposition}

\begin{proof}
Let $x\in S_X$ be denting. Then for every $\varepsilon>0$, there is a functional $x^*\in S_{X^*}$ and some $0<\delta<\varepsilon$ such that $x\in S_1:=S(B_X, x^*, \delta)$ and $\diam(S_1)<\varepsilon$. Consider $f^*:=(x^*, 0^*)\in S_{X^*\oplus_\infty Y^*}$, and let $C:=\frac{1}{2} S_1 + \frac{1}{2} S_2$, where $S_1:=S(B_{X\oplus_1 Y}, f^*, \delta)$ and $S_2:=S(B_{X\oplus_1 Y}, -f^*, \delta)$. Note that $(0,0)\in C$. 

Let $(z,t)\in C$. There exist $(z_1, t_1)\in S_1$ and $(z_2, t_2)\in S_2$ such that $(z,t)=\frac{1}{2}(z_1, t_1) + \frac{1}{2}(z_2, t_2)$. In particular,
$$\re(x^*(z_1)) = \re(f^*((z_1, t_1)) > 1-\delta\quad \text{and}\quad \re(x^*(z_2))=\re(f^*((z_2, t_2)) < -1+\delta.$$
This implies in particular that $\min\{\|z_1\|, \|z_2\|\}\geq 1-\delta$, so $\max\{\|y_1\|, \|y_2\|\}\leq \delta<\varepsilon$. Finally, note that since $z_1\in S(B_X, x^*, \delta)$, we have $\|x-z_1\|\leq \varepsilon$, and similarly $\|-x-z_2\|\leq \varepsilon$ by symmetry, so
$$\|(z,t)\|=\frac{1}{2}(\|z_1+z_2\| + \|t_1+t_2\|)\leq \frac{1}{2}(\|z_1 - x\| + \|x+z_2\| + \|t_1\| + \|t_2\|) < 2\varepsilon.$$
This shows that $\diam(C)\leq 2\varepsilon$.
\end{proof}

However, the situation is different for ccs-$\Delta$ points, where the following upwards stability result is true regardless of the involved spaces.

\begin{lemma}\label{ccs-Delta-1-sums}
Let $X,Y$ be two Banach spaces and $x\in S_X$. If $x$ is a ccs-$\Delta$ point, then $(x,0)\in S_{X\oplus_1 Y}$ is also a ccs-$\Delta$ point.
\end{lemma}

\begin{proof}
Suppose that $x$ is a ccs-$\Delta$ point and fix $\varepsilon>0$. Let $\widetilde{C}=\sum_{i=1}^n \lambda_i\widetilde{S_i}$ be a convex combination of slices of $B_{X\oplus_\infty Y}$ containing $(x,0)$ where for each $1\leq i\leq n$, $\widetilde{S_i}:=S(B_{X\oplus_1 Y}, f_i^*, \delta_i)$ for some $f_i^*=(x_i^*, y_i^*)\in S_{X^*\oplus_\infty Y^*}$ and $\delta_i>0$. For each $1\leq i\leq n$, there exist $(z_i,t_i)\in \widetilde{S_i}$ such that $(x,0)=\sum_{i=1}^n \lambda_i (z_i, t_i)$. But since $x$ has norm $1$ and $\sum_{i=1}^{n} \lambda_i=1$, we must have $\|z_i\|=1$, and so, $t_i=0$, for all $1\leq i\leq n$. Now, for each $1\leq i\leq n$, set $S_i:=S(B_X, x_i^*, \delta_i)$. Notice that
$$\re(f^*((z_i, t_i))) = \re(x_i^*(z_i)) > 1-\delta_i.$$
This means that $x$ belongs to the convex combination of slices $C=\sum_{i=1}^n \lambda_i S_i$ of $B_X$, so there must be some $z\in C$ such that $\|x-z\|\geq 2-\varepsilon$. Note now that
$$\re(x_i^*(z)) = \re(f^*((z,0)))>1-\delta_i,$$
so $(z,0)\in \widetilde{C}$, and clearly we have that
$$\|(z,0)-(x,0)\|=\|z-x\|\geq 2-\varepsilon.$$
This shows that $(x,0)$ is a ccs-$\Delta$ point.
\end{proof}

We will use this result in a later section to extend some conclusions from \cite[Section 4]{MPZpre}. Also, both of the previous results can be adapted to other $\oplus_p$-sums if $1\leq p < \infty$.

% \textcolor{red}{Remark: The reason why we couldn't get downwards stability here even for the case $x\mapsto (x,0)$ is that in the \textit{straightforward argument}, even if $(z,t)$ is far from $(x,0)$, that does not mean in general that $z$ is far from $x$. There is also a problem when trying to get downwards stability results for super-$\Delta$ and super-Daugavet points.}

\subsection{Final remarks on stability results}\label{Subsection:stability-final-remarks}

Stability results can be used to study the points in spaces that can be split as a direct sum $Z=X\oplus Y$. In particular, upwards stability results provide sufficient conditions to get such points $(x,y)\in S_{X\oplus Y}$, whereas downwards stability results provide necessary conditions for such points to exist. Also, note the following.

\begin{remark}\label{remark:uses-of-stability}
The stability results associating $x\in S_X$ with $(x,0)\in S_{X\oplus Y}$ can be used as in \cite[Remark 4.4]{MPZpre} to extend the conclusions stated there. More specifically, as a consequence of all these stability lemmas, we have the following:
\begin{enumerate}
\item If two notions among Daugavet, $\Delta$, super-Daugavet, super-$\Delta$, and ccs-Daugavet can be separated in a Banach space $X$, then they are also distinct in any space that can be expressed as $Z=X\oplus_\infty Y$ for some Banach space $Y$. This is the case for $C_0(L,X)$ when $L$ has isolated points, for $A(K,X)$ when $\Gamma$ has isolated points, and for $L_\infty(\mu, X)$ when $\mu$ has atoms. This provides a big contrast between the scalar-valued and the vector-valued versions of these spaces.
\item If $\Delta$ and Daugavet points can be separated in a Banach space $X$, then they are also distinct in any spaces that can be expressed as $Z=X\oplus_1 Y$ for some Banach space $Y$. This is the case for $L_1(\mu, X)$ spaces when $\mu$ has atoms.
\end{enumerate}
\end{remark}

\section{Main results}\label{section-results}

In this section we will study the point notions in several classes of Banach spaces. Given a set $A$ in some topological space, we will always denote by $A'$ the the set of its limit points.

\subsection{Results on \texorpdfstring{$\boldsymbol{C_0(L,X)}$}{C0(L,X)} spaces}\hfill\\ \label{subsection:C0LX}

Let $K$ be a compact Hausdorff space and let $X$ be a Banach space. As mentioned in the introduction, it is known (see \cite[Corollary 4.3]{MPZpre}) that for $f\in S_{C(K)}$, the first six diametral point notions from Definition \ref{def-six-notions} are equivalent, and in fact $f$ is any of those types of points if and only if it attains its maximum (norm) at a limit point, that is, if there exists $t\in K'$ such that $|f(t)|=\|f\|_\infty$. If $C(K)$ is a complex space or an infinite-dimensional space, then those notions also coincide with $\nabla$ points by Proposition \ref{prop-nablas-complex} (note however that this is no longer true in the real case, as for all $n\in\bbn$, the real space $\ell_\infty^n$ has $\nabla$ points, see \cite{HLPV}). All these results can also be extended to $C_0(L)$ spaces where $L$ is a locally compact Hausdorff space, as we will see (see also \cite[Proposition 3.9 and Remark 3.10]{HLPV}). 

However, the situation is different in the vector-valued case. For $f\in S_{C(K,X)}$, attaining the maximum norm at a limit point implies being a ccs-Daugavet point (see \cite[Theorem 4.2]{MPZpre}), but being $\Delta$ does not imply this condition in general, as pointed out in \cite[Remark 4.4]{MPZpre} (also see Remark \ref{remark:uses-of-stability}). In this section we will study the point notions in the more general setting of $C_0(L,X)$ spaces for a locally compact Hausdorff $L$. 

In the next result, we fully characterize Daugavet, $\Delta$, and $\nabla$ points of $C_0(L,X)$ in terms of $L$ and $X$. We also characterize super-Daugavet, super-$\Delta$, and ccs-Daugavet points in terms of $L$ and finite $\oplus_\infty$-sums of $X$, and provide some conditions about ccs-$\Delta$ points.

\begin{theorem}\label{theo-charac-Daug-C0LX}
Let $L$ be a locally compact Hausdorff topological space, let $X$ a Banach space, and let $f\in S_{C_0(L, X)}$.
\begin{enumerate}
\item If $f$ attains its maximum norm at some non-isolated point $t_0\in L'$ (that is, if $\|f(t_0)\|=\|f\|_\infty$), then $f$ is a ccs-Daugavet point.
\item Otherwise, let $T:=\{t_1,\ldots,t_n\}\subset L$ ($n\in\bbn$) be the collection of isolated points $t_k$ where $f$ attains its maximum norm, which is finite and non-empty. We have the following.
\begin{enumerate}
\item $f$ is super-Daugavet (resp. super-$\Delta$, or ccs-Daugavet) if and only if $f|_T$ is super-Daugavet (resp. super-$\Delta$, or ccs-Daugavet) in $C_0(T,X)=X\oplus_\infty\stackrel{n)}{\cdots}\oplus_\infty X$.
\item $f$ is Daugavet if and only if there is some $t_0\in T$ such that $f|_{\{t_0\}}$ is Daugavet in $C_0(\{t_0\}, X)=X$.
\item If $f$ is ccs-$\Delta$, then $f|_T$ is ccs-$\Delta$ in $C_0(T,X)=X\oplus_\infty\stackrel{n)}{\cdots}\oplus_\infty X$.
\item $f$ is $\Delta$ if and only if either there is some $t_k\in T$ such that $f|_{\{t_k\}}$ is $\Delta$ in $C_0(\{t_k\}, X)=X$, or for each $\{p_k\}_{k=1}^n\subset (1,+\infty)$ with $\sum_{k=1}^n \frac{1}{p_k}=1$, there exists some $k\in\{1,\ldots,n\}$ such that $f|_{\{t_k\}}$ is $\Delta_{p_k}$ in $C_0(\{t_{k}\}, X)=X$. 
\item If $L\neq T$, or if $C_0(L,X)$ is a complex Banach space, $f$ is $\nabla$ if and only if $f$ is Daugavet.
\item If $L=T$, then $f$ is $\nabla$ if and only if either $f|_{\{t_i\}}$ is $\nabla$ in $C_0(\{t_i\}, X)=X$ for all $1\leq i\leq n$, or there is some $t_0\in T$ such that $f|_{\{t_0\}}$ is Daugavet in $C_0(\{t_0\}, X)=X$ (and in this case, $f$ is Daugavet as well).
\end{enumerate}
\end{enumerate}
\end{theorem}

\begin{proof}
We show all the cases.
\begin{enumerate}\setlength{\parindent}{0pt}
\item We adapt the proof of \cite[Theorem 4.2]{MPZpre}. Let $x^* \in S_{X^*}$ such that $\operatorname{Re}\,x^*(f(t_0)) = \|f(t_0)\|_X = 1$ and $C:= \sum_{i=1}^{m} \lambda_i S_i$ be a convex combination of slices of $B_{C_0(L, X)}$. Pick $g_i \in S_i$ for every $i = 1, \cdots m$. Let $K$ be a compact subset of $L$ that contains $t_0$ such that $\|f(t)\|_X < \epsilon$ for all $t \in L \setminus K$. Since $t_0$ is a limit point in $L$, we claim that there exists a sequence $(U_n)_{n=0}^{\infty}$ of open subsets containing $t_0$ that satisfies the following properties:
\begin{enumerate}
\item[(i)] $U_0 = L$
\item[(ii)] $\overline{U_{n+1}}$ is a compact, proper subset of $U_n$ for every $n \geq 0$
\item[(iii)] $\operatorname{Re}\,x^*(f(t)) > 1 - \frac{1}{n+1}$ and $\|g_i(t) - g_i(t_0)\|_X < \frac{1}{n+1}$ for every $t \in \overline{U_n}$, $i \in \{1, \cdots m\}$, and $n \geq 1$.
\end{enumerate} 
We construct this by an inductive process. Let $L = U_0$. Since $K \subset U_0$ is normal and has non-empty interior, there exists an open subset $V_1$ containing $t_0$ such that $V_1 \subset \overline{V_1}\subset K \subset U_0$. Then by normality, there exists an open subset $W_1$ containing $t_0$ such that $\overline{W_1}$ is a proper subset of $\overline{V_1}$. By the continuity of $x^*\circ f$ and $g_i$'s, there exists an open subset $U_1$ of $W_1$ containing $t_0$ such that $\operatorname{Re}\,x^*(f(t)) > \frac{1}{2}$ and $\|g_i(t) - g_i(t_0)\|_X < \frac{1}{2}$ for every $t \in \overline{U_1}$ and $i \in \{1, \cdots m\}$. Next, notice that $U_1$ also contains a compact neighborhood $K_1$ of $t_0$, since $L$ is locally compact. Then there exists an open subset $V_2$ containing $t_0$ such that $V_2 \subset \overline{V_2}\subset K_1 \subset U_1$. Hence we can also find an open subset $W_2$ containing $t_0$ such that $\overline{W_2}$ is a proper subset of $\overline{V_2}$, and so there exists an open subset $U_2$ containing $t_0$ such that $\operatorname{Re}\,x^*(f(t)) > \frac{2}{3}$ and $\|g_i(t) - g_i(t_0)\|_X < \frac{1}{3}$ for every $t \in \overline{U_2}$ and $i \in \{1, \cdots m\}$. By inductively continuing this procedure, we obtain the desired sequence $(U_n)_{n=1}^{\infty}$ of open subsets containing $t_0$ that satisfy the conditions (i), (ii), and (iii).

% \textcolor{red}{Comment to self: re-read last paragraph carefully to double check.}

Now, for each $n\in\bbn$, define $F_n = \overline{U_{n+1}} \setminus U_{n+2}$. Each $F_n$ is not only closed but also compact. Moreover, $F_n \subset U_{n} \setminus \overline{U_{n+3}} \subset L$. Hence, there exists $\phi_n$ such that $0 \leq \phi_n(t) \leq 1$ where $\phi_n(t) = 1$ on $t \in F_n$ and $\phi_n(t) = 0$ on $t \in (L\setminus U_n) \cup \overline{U_{n+3}}$. Notice that the sequence $(\phi_n)_{n=1}^{\infty} \subset S_{C_0(L)}$ converges pointwise to $0$. For each $i\in\{1,\ldots,m\}$, define
\[
g_{i, n}(t): = \frac{n}{n+1}(g_i(t) - (f(t_0) + g_i(t_0))\phi_n(t)).
\]
For every $t \in (L \setminus U_n) \cup \overline{U_{n+3}}$, $\|g_{i,n}(t)\|_X = \frac{n}{n+1}\|g_i(t)\|_X < 1$. On the other hand, for $t \in U_n\setminus \overline{U_{n+3}}$ we have
\[
\|g_{i,n}(t)\|_X  = \frac{n}{n+1}\|g_i(t) - (f(t_0)) + g_i(t_0))\phi_n(t)\|_X \leq \frac{n}{n+1}(1 + \|g_i(t) - g_i(t_0)\|_X) \leq 1.
\]  
Thus, each $g_{i,n} \in B_{C_0(L, X)}$ and so $g_{i,n}$ converges weakly to $g_i$ with $n$ for each $i = 1, \cdots, m$. From the fact that slices are weakly open subsets, there exists $N$ such that for every $n \geq N$, $g_{i,n} \in S_i$. Furthermore, the function $h_n: = \sum_{i=1}^{m}\lambda_i g_{i,n}$ where $\sum_{i=1}^{m}\lambda_i = 1$ belongs to $C$ for every $n \geq N$. 

Let $t \in F_{n} \subset \overline{U_{n+1}} \subset \overline{U_n}$. From the condition (iii) we have $\operatorname{Re}\,x^*(f(t)) \geq 1 - \frac{1}{n+1}$. Hence,

\begin{eqnarray*}
\operatorname{Re}\,x^*(g_{i,n}(t)) &=& \frac{n}{n+1}(\operatorname{Re}\,x^*(g_i(t)) - 1 - \operatorname{Re}\,x^*(g_i(t_0)))\\ &\leq& \frac{n}{n+1}(\operatorname{Re}\,x^*(g_i(t)) - \operatorname{Re}\,x^*(g_i(t_0)) + \operatorname{Re}\,x^*(g_i(t_0)) - 1 - \operatorname{Re}\,x^*(g_i(t_0)))\\
&\leq& \frac{n}{n+1}\left(\operatorname{Re}\,x^*(g_i(t_0)) -1 - \operatorname{Re}\,x^*(g_i(t_0))\right) + \frac{1}{n+1}\\
&=& -1 + \frac{2}{n+1}.
\end{eqnarray*}
This implies that
\[
\|f - h_n\| \geq \operatorname{Re}\,x^*\left(f(t) - \sum_{i=1}^{m}\lambda_ig_{i,n}(t)\right) \geq 1 - \frac{1}{n+1} + 1 - \frac{2}{n+1} = 2 - \frac{3}{n+1},
\]
and so $\|f - h_n\| \rightarrow 2$ as $n \rightarrow \infty$. The proof is finished.

\item Let $T$ be given, and note that $T$ is necessarily finite and non-empty. If $T=\{t_1, \ldots, t_n\}\subset L$ for some $n\in\bbn$, note that $C_0(T, X)=X\oplus_\infty \stackrel{n)}{\cdots} \oplus_\infty X$, and in fact, for each $t\in L$, $C_0(\{t\}, X)=X$ isometrically. Note also that if $L\neq T$, then $f=(f|_{L\setminus T}, f|_T)$ in $C_0(L, X)=C_0(L\setminus T, X)\oplus_\infty C_0(T, X)$, and $\|f|_{L\setminus T}\|<1$ by assumption. We can now finish the proof in view of the stability lemmas.

For \textup{(a)}, it suffices to apply Lemmas \ref{super-infty-sums} and \ref{ccs-Daugavet-infty-sums}. \textup{(b)} follows from Lemma \ref{daugavet-infty-sums}. \textup{(c)} follows from Lemma \ref{ccs-Delta-infty-sums}. \textup{(d)} follows from Lemma \ref{delta-infty-sums}. Finally, for $\nabla$ points, the complex case is shown in Proposition \ref{prop-nablas-complex}, and in the real case, items \textup{(e)} and \textup{(f)} follow from Lemma \ref{nabla-infty-sums}. \qedhere
% Finally, let us see \textup{(d)}. By Lemma \ref{lemma:stability-infty-down-2}, $f$ is $\Delta$ if and only if $f|_T$ is. Now we apply Lemma \ref{lemma:stability-infty-down-1} to $f|_T$, and there are two possibilities:
% \begin{itemize}
% \item If there is $t\in T$ such that $f|_{\{t\}}$ is $\Delta$ in $C_0(\{t\}, X)$, then it is $\Delta_p$ for all $p>1$, and $f$ is $\Delta$ by Lemma \ref{lemma:stability-infty-up}.
% \item Otherwise, we just apply \cite[Propositions 4.4 and 4.5]{HPV21}. \qedhere
% \end{itemize}
\end{enumerate}
\end{proof}

% \textcolor{red}{Do we need more details?}

% \color{red}
% \begin{remarks}\hfill\\
% \begin{itemize}
% \item As we do not have the corresponding downwards stability results for $(x,y)\in S_{X\oplus_\infty Y}$ with $\|x\|=\|y\|=1$ as in Daugavet points, we do not know if in (a) above we can get a conclusion similar to that of (b) unless $|T|=1$.
% \item For (c), we do not know if the converse of Lemma \ref{ccs-Delta-infty-sums}.(1) holds, so we can only get an implication in this case. 
% \item In view of Remark \ref{remark:uses-of-stability} and the results and comments from Section \ref{section:stability}, we can separate most notions in some $C_0(L,X)$ spaces, unlike what happens in the scalar-valued $C_0(L)$ spaces.
% \end{itemize}
% \end{remarks}
% \color{black}

An immediate consequence of the theorem is the following well-known fact: $C_0(L,X)$ has Daugavet property if and only if either $L=L'$ or $X$ has Daugavet property. For the backward implication, just use the previous theorem. For the forward implication, it suffices to apply the following well-known lemma, which can be derived from Lemma \ref{daugavet-infty-sums}.

\begin{lemma}\label{lemma-Daugavet-Down-infty-sum}
Given two Banach spaces $X$ and $Y$, if $Z=X\oplus_\infty Y$ has Daugavet property, then so does $X$.
\end{lemma}

% \begin{proof}
% Suppose that $Z=X\oplus_\infty Y$ has Daugavet property. Let $T:X\rightarrow X$ be a rank-one operator. Let $\widetilde{T}:Z\rightarrow Z$ be given by $\widetilde{T}(x,y):= (T(x), 0)$ for each $x\in X$ and each $y\in Y$. Clearly $\|\widetilde{T}\|=\|T\|$ and $\widetilde{T}$ has rank one, so $\|\widetilde{\id}+\widetilde{T}\|=1+\|T\|$, where $\widetilde{\id}$ denotes the identity operator on $Z$. Let $\id$ denote the identity operator on $X$. Then we have
% $$\sup_{\|(x,y)\|_\infty=1}\|\widetilde{\id}(x,y)+\widetilde{T}(x,y)\|_\infty = \sup_{\|(x,y)\|_\infty=1} \max\{\|x+T(x)\|_\infty, \|y\|_\infty\} = 1+\|T\|.$$
% Therefore, there is a sequence of points $\{x_n\}_{n=1}^{\infty}\subset B_X$ such that $\|x_n + T(x_n)\|_\infty\rightarrow 1+\|T\|$, so $\|\id+T\|=2$. This shows that $X$ has Daugavet property.
% \end{proof}

Recall that a Banach space $X$ is said to be \textit{locally uniformly rotund} (or just LUR) if 
$$\lim_k \|x_k-x\|=0\quad \text{whenever}\quad \lim_k \left\|\frac{1}{2} (x_k+x)\right\|=\lim_k \|x_k\|=\|x\|.$$
Equivalently, a Banach space is LUR if and only if for every $\varepsilon>0$ and every $x\in S_X$, the \textit{modulus of local rotundity at $x$}
$$\delta_X(x, \varepsilon):=\inf\left\{ 1-\left\|\frac{x+y}{2}\right\|:\, y\in B_X,\, \|x-y\|\geq\varepsilon \right\}$$
is positive.

Theorem \ref{theo-charac-Daug-C0LX} reinforces once more the claim derived from Remark \ref{remark:uses-of-stability}: unlike in the scalar-valued case, in $C_0(L,X)$ spaces the diametral point notions do not coincide in general. However, the situation changes if $X$ is LUR. In fact, we have the following.

\begin{theorem}\label{theo:implications-C0LX}
Let $L$ be a locally compact Hausdorff space, and let $X$ be a Banach space. Given $f\in S_{C_0(L,X)}$, consider the following diagram of implications
\[
\begin{tikzcd}
    \begin{array}{c} \|f(t)\|=\|f\|_\infty\\ \text{ for some }t\in L' \end{array} \\
    {f\text{ ccs-Daugavet}} & {f\text{ super-Daugavet}} & {f\text{ Daugavet}} & {f\text{ }\nabla} \\
    {f\text{ ccs-}\Delta} & {f\text{ super-}\Delta} & {f\text{ }\Delta} & \begin{array}{c} \|f(t)\|=\|f\|_\infty\\ \text{ for some }t\in L' \end{array}
    \arrow["{(1)}", from=1-1, to=2-1]
    \arrow[from=2-1, to=2-2]
    \arrow[from=2-1, to=3-1]
    \arrow[from=2-2, to=2-3]
    \arrow[from=2-2, to=3-2]
    \arrow[from=2-3, to=2-4]
    \arrow[from=2-3, to=3-3]
    \arrow["{(4)}", from=2-3, to=3-4]
    \arrow["{(3)}", from=2-4, to=3-4]
    \arrow[bend right=20, from=3-1, to=3-3]
    \arrow[from=3-2, to=3-3]
    \arrow["{(2)}"', from=3-3, to=3-4]
\end{tikzcd}
\]
Then:
\begin{enumerate}
\item[(i)] The unnumbered implications are always true in every Banach space.
\item[(ii)] The implication \textup{(1)} always holds in $C_0(L,X)$.
\item[(iii)] While the implication \textup{(2)} is not always true in general, it is true if $X$ is LUR.
\item[(iv)] The implication \textup{(3)} holds for every $f\in S_{C_0(L,X)}$ if and only if either $L=L'$ or $X$ has no $\nabla$ points.
\item[(v)] The implication \textup{(4)} holds for every $f\in S_{C_0(L,X)}$ if and only if either $L=L'$ or $X$ has no Daugavet points.
\end{enumerate}
\end{theorem}

\begin{proof}
Note that \textup{(i)} is well known, and \textup{(ii)} has been shown already in Theorem \ref{theo-charac-Daug-C0LX}. For \textup{(iii)}, given that $X$ is LUR, assume to the contrary that $f \in S_{C_0(L, X)}$ is a $\Delta$-point that does not attain its maximum norm at any limit point of $L$. Denote $F = \{t \in L : \|f(t)\|_X = 1\}$. The set $F$ contains only isolated points of $L$ and is nonempty. Fix $\epsilon_0 > 0$. Since $L$ is locally compact, we can find a compact subset $K:= K_f \subset L$ that contains $F$ and such that $\|f(t)\|_X < \epsilon_0$ for all $t \in L \setminus K$. Because of the compactness of $K$, the set $F$ must be finite (otherwise, there exists a limit point in $F$ where $f$ attains its norm). Also, for each $t \in F$, we can find $x_t^* \in S_{X^*}$ such that $x_t^*(f(t)) = 1$ by the Hahn-Banach theorem. Since $X$ is LUR, take $\delta:= \inf\left\{\delta_X\left(f(t), \frac{1}{2|F|}\right) : t \in F\right\}>0$.

% , where 
% $$\delta_X(x, \varepsilon):=\inf\left\{ 1-\left\|\frac{x+y}{2}\right\|:\, y\in B_X,\, \|x-y\|\geq\varepsilon \right\}$$
% \textcolor{red}{is the modulus of local rotundity at $x$.}

If $\operatorname{Re}\,x_t^*(x) \geq 1 - \delta$ for $x \in B_X$, then 
\[
\left\|\frac{x + f(t)}{2}\right\|_X \geq \operatorname{Re}\,x_t^*\left(\frac{x + f(t)}{2}\right) \leq 1 - \frac{\delta}{2} < 1 - \delta_X\left(f(t), \frac{1}{2|F|}\right). 
\]
This implies that $\|x - f(t)\|_X < \frac{1}{2|F|}$. 

Let $\epsilon = 1 - \max_{t \in L \setminus F}\{\|f(t)\|_X, \epsilon_0\} > 0$ and define a bounded linear functional $\psi \in C_0(L, X)^*$ by 
\[
\psi(g) = \frac{1}{|F|}\sum_{t \in F}x_t^*(g(t)).
\]
Since $\psi(f) = 1$, we have $\|\psi\| = 1$. With this bounded linear functional, we define a norm-one, rank-one projection $P: C_0(L, X) \rightarrow C_0(L, X)$ by $P(g):= (\psi \otimes f)(g) = \psi(g) \cdot f$. 

To finish the proof, we show that the projection $P$ satisfies $\|I - P\| < 2$. For a given $g \in B_{C_0(L, X)}$, if $t \in L \setminus K$ we have 
\[
\|g(t) - Pg(t)\|_X = \|g(t) - \psi(g)(t)f(t)\|_X \leq 1 + \|f(t)\|_X < 2 - \epsilon.
\]
For $t \in K$ we consider two cases. First, if $|x_{t'}^*(g(t'))| \geq 1 - \delta$ for $t' \in F$ there exists $\alpha_{t'} \in S_{\mathbb{C}}$ such that $|x_{t'}^*(g(t'))| =  x_{t'}^*(\alpha_{t'}g(t')) \geq 1 - \delta$. Hence, $\|\alpha_{t'} g(t') - f(t')\|_X \leq \frac{1}{2|F|}$. So we obtain
\begin{eqnarray*}
\|g(t') - Pg(t')\|_X  &=& \left\|g(t') - \frac{1}{|F|}\sum_{t \in F}x_t^*\psi(g)(t')f(t')\right\|_X \\
&\leq& \left\|g(t') - \frac{1}{|F|}x_{t'}^*g(t')f(t')\right\|_X + \frac{|F|-1}{|F|}\\
&=& \left\|\alpha_{t'}g(t') - \frac{1}{|F|}x_{t'}^*(\alpha_{t'}g(t'))f(t')\right\|_X + \frac{|F|-1}{|F|}\\
&\leq& \|\alpha_{t'}g(t') - f(t')\|_X + \left\|f(t') - \frac{1}{|F|}x_{t'}^*(\alpha_{t'}g(t'))f(t')\right\|_X + \frac{|F|-1}{|F|}\\
&\leq& \frac{1}{2|F|} + 1 - \frac{1}{|F|}x_{t'}^*(\alpha_{t'}g(t')) + \frac{|F|-1}{|F|}\\
&\leq& 2 - \frac{1}{2|F|}.
\end{eqnarray*}

If $|x_{t'}^*g(t')| < 1- \delta$ for $t' \in F$, then we have 
\begin{eqnarray*}
\|g(t') - Pg(t')\|_X  &=& \left\|g(t') - \frac{1}{|F|}\sum_{t \in F}x_t^*\psi(g)(t')f(t')\right\|_X \\
&\leq& \left\|g(t') - \frac{1}{|F|}x_{t'}^*g(t')f(t')\right\|_X + \frac{|F|-1}{|F|}\\
&=& \left\|\alpha_{t'}g(t') - \frac{1}{|F|}x_{t'}^*(\alpha_{t'}g(t'))f(t')\right\|_X + \frac{|F|-1}{|F|}\\    
&<& 1 + \frac{1}{|F|}(1 - \delta) + \frac{|F|-1}{|F|} = 2 - \frac{\delta}{|F|}.
\end{eqnarray*}
Hence, we obtain $\|g - Pg\| \leq \max\{2 - \frac{1}{2|F|}, 2 - \frac{\delta}{|F|}, 2 - \epsilon\} < 2$ for every $g \in B_{C_0(L, X)}$, and so $\|I-P\| < 2$. However, since $f$ is a $\Delta$-point, this cannot happen by \cite[Lemma 2.1]{AHLP} (for complex Banach spaces, see \cite[Theorem 2.2]{LT24}). Therefore, $f$ must attain the norm at a limit point of $L$. 

For \textup{(iv)} and \textup{(v)} note that if $L=L'$, then $C_0(L,X)$ has Daugavet property, so every point is Daugavet. Otherwise, you can write $C_0(L, X)=X\oplus_\infty Y$ for some Banach space $Y$. Now the conclusion follows from Lemmas \ref{nabla-infty-sums} and \ref{daugavet-infty-sums}.
\end{proof}

Recall that a Banach space $X$ is said to have the \textit{polynomial Daugavet property} if every weakly compact polynomial $P \in \mathcal{P}(X; X)$ satisfies the Daugavet equation $\|\id+P\|=1+\|P\|$. If $X$ has the polynomial Daugavet property, then the space also has the Daugavet property. It is known that $C_0(L, X)$ has the polynomial Daugavet property if and only if either $X$ does or $L=L'$ (see \cite[Proposition 6.10]{CGMM08}). Thus, as a consequence of Theorem \ref{theo:implications-C0LX}, we get the following.

\begin{Corollary}\label{cor:c0lxequiv}
Let $L$ be a locally compact Hausdorff space and $X$ be a LUR Banach space. The following statements are equivalent:
\begin{enumerate}
\item $C_0(L, X)$ has the polynomial Daugavet property.
\item $C_0(L, X)$ has the Daugavet property
\item $C_0(L, X)$ has the restricted DSD2P.
\item $C_0(L, X)$ has the DD2P.
\item $C_0(L, X)$ has the DLD2P.
\item $C_0(L, X)$ has the property ($D$).
\item $L$ does not contain isolated points.
\end{enumerate}
\end{Corollary}

\begin{proof}
Note that $\textup{(7)}\Rightarrow\textup{(1)}\Rightarrow\textup{(2)}\Rightarrow\textup{(3)}\Rightarrow\textup{(4)}\Rightarrow\textup{(5)}\Rightarrow\textup{(6)}$ were already known. Finally, that $\textup{(6)}\Rightarrow \textup{(7)}$ is implicitly shown in the proof of item \textup{(iii)} from Theorem \ref{theo:implications-C0LX}, as the operator $P$ that leads to the contradiction is a norm-one rank-one projection.
\end{proof}

\subsection{Results on \texorpdfstring{$\boldsymbol{A(K,X)}$}{A(K,X)} spaces}\hfill\\ \label{subsection:AKX}

Let $K$ be a compact Hausdorff topological space, and let $X$ be a Banach space. It was shown in \cite[Corollary 5.5]{LT22} that a point in a uniform algebra $f\in S_{A(K)}$ being Daugavet and $\Delta$ is equivalent to attaining the maximum norm at a limit point of the Shilov boundary. In the case of $A(K,X)$ spaces, a point on $S_{A(K, X)}$ attaining the maximum norm at a limit point of the Shilov boundary of the base algebra $A(K)$ is Daugavet, but the point being $\Delta$ does not imply this condition in general (since it does not for $C(K,X)$). However, if $X$ is uniformly convex (see \cite[Theorem 5.4]{LT22}), this implication holds. Also, note that uniform algebras over an infinite compact $K$ do not have strongly exposed points (see \cite[Main Theorem]{BW99}), so in this case $\nabla$ and Daugavet notions coincide (see \cite[Theorem 2.3]{HLPV} and also Proposition \ref{prop-nablas-complex}).

In this section we will see that all the results from Subsection \ref{subsection:C0LX} hold analogously for $A(K,X)$ spaces. To do so, we need some preliminary results. First, note that if the Shilov boundary of the base algebra has an isolated point, then we can write $A(K,X)=Y\oplus_\infty X$ for some Banach space $Y$. This is a consequence of the fact that characteristic functions of isolated points always belong to a uniform algebra.

\begin{lemma}[{\cite[Lemma 2.5]{LT22}}]
Let $A(K)$ be a uniform algebra on a compact Hausdorff space $K$ and let $t_0$ be an
isolated point of the Shilov boundary $\Gamma$ of A. Then there exists a function $\phi\in A(K)$ such that
$$\phi(t_0) = \|\phi\|=1\text{ and }\phi(t)=0\text{ for }t\in \Gamma\setminus \{t_0\}.$$
\end{lemma}

We need one more tool. Note that the proof of \cite[Theorem 4.2]{MPZpre}, like many other geometrical properties of $C(K)$ spaces from the literature, heavily relies on using the Urysohn's lemma. However, in the context of uniform algebras $A(K)$, this lemma does not hold in general, since the functions generated by the lemma may not be in $A(K)$ (see \cite{CGK13}). To adapt Theorem \ref{theo-charac-Daug-C0LX} to $A(K,X)$ spaces we will use the following Urysohn-type lemma from \cite{CGK13}.

\begin{Lemma}[{\cite[Lemma 2.5]{CGK13}}]\label{lemma:urysohn-type}
Let $K$ be a compact Hausdorff topological space, $A_u(K)$ be a (real or complex) uniform algebra over $K$, and let $\Gamma_0$ be its Choquet boundary. Then, for every open set $U\subset K$ with $U\cap \Gamma_0\neq \emptyset$ and $0<\varepsilon<1$, there exists $f\in A_u(K)$ and $t_0\in U\cap \Gamma_0$ such that $f(t_0)=\|f\|_\infty=1$, $|f(t)|<\varepsilon$ for every $t\in K\setminus U$, and for all $t\in U$, we have $|f(t)|+(1-\varepsilon)|1-f(t)|\leq 1$.
\end{Lemma}

In order to apply that lemma for this setting, we need a geometric result for which we provide an elementary proof.

\begin{lemma}\label{lemma:stolzing}
Let $\varepsilon\in (0,1)$. Then for every complex number $x$ in the Stolz region
$$\operatorname{St}_\varepsilon:=\{z\in \bbc:\, |z| + (1-\varepsilon) |1-z|\leq 1\},$$
there is some $y\in [0,1]\subset \bbc$ such that
$$|x -y|\leq \sqrt{\frac{1}{4(1-\varepsilon)^2}-\frac{1}{4}}.$$
\end{lemma}

\begin{proof}
Let $\varepsilon\in (0,1)$, and let $x\in \ste$. First note that, trivially, the region $\ste$ is entirely contained in the ellipse given by
$$\left\{z\in \bbc:\, (1-\varepsilon)|z| + (1-\varepsilon)|1-z|\leq 1\right\},$$
that is,
$$\left\{z\in \bbc:\, |z-0| + |z-1|\leq \frac{1}{1-\varepsilon} \right\},$$
with foci on the points $0$ and $1$. Since the major semiaxis measures
$$M_\varepsilon:=\frac{1}{2\left( 1-\varepsilon \right)},$$
the minor semiaxis measures
$$m_\varepsilon:= \sqrt{\left( \frac{1}{2\left( 1-\varepsilon \right)} \right)^2 - \left( \frac{1}{2} \right)^2}.$$
Note that if $a,b,a-b,a+b>0$, then clearly
$$\sqrt{a^2-b^2}=\sqrt{a-b} \sqrt{a+b} > a-b,$$
so in particular,
$$m_\varepsilon>M_\varepsilon-\frac{1}{2}.$$
Therefore, this ellipse is entirely contained in the convex hull $E:=\overline{\conv}(B_1\cup B_2)$, where $B_1=B(0, m_\varepsilon)$ and $B_2=B(1, m_\varepsilon)$ are balls with radius $m_\varepsilon$ centered on the foci (see Figure \ref{fig:stolz-ellipse-2}). This shows that there exists $y\in [0,1]$ such that $|x-y|\leq m_\varepsilon$, and in fact this point can be taken explicitly as follows
$$\begin{cases}
y=0,\quad &\text{if }\re(x)<0,\\
y=\re(x),\quad &\text{if }\re(x)\in [0,1],\\
y=1,\quad &\text{if }\re(x)>1.
\end{cases}$$
\begin{figure}[H]
\centering
\includegraphics[width=8cm]{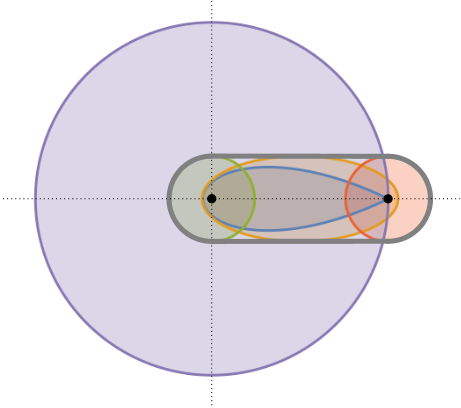}
\caption{Stolz region (blue) is contained in the ellipse (orange), which is contained in the convex hull (gray) of the small balls centred at the foci (green and red).\qedhere}
\label{fig:stolz-ellipse-2}
\end{figure}
\end{proof}

% Regarding the analogous result of Theorem \ref{theo:implications-C0LX} for $A(K,X)$ spaces, note that the proof of \cite[Theorem 4.2]{MPZpre} relies heavily on Urysohn's Lemma, which is not true in general in uniform algebras. However, there is an Urysohn-type lemma from \cite{CGK13} that we can carefully apply.

% Many geometrical properties of $C(K)$ spaces from the literature

% In the literature, studying geometrical properties on $C(K)$ space often relies on using the Urysohn lemma. However, this lemma is inadequate to study the geometrical properties of uniform algebras $A(K)$ and their related spaces because the functions generated by the lemma may not be in $A(K)$ (see \cite{CGK13}). Hence, we will use the Urysohn-type functions whose existence is guaranteed by the following lemma. 

% \begin{Lemma}[{\cite[Lemma 2.5]{CGK13}}]\label{lemma:urysohn-type}
% Let $K$ be a compact Hausdorff topological space, $A(K)$ be a uniform algebra over $K$, and let $\Gamma_0$ be its Choquet boundary. Then, for every open set $U\subset K$ with $U\cap \Gamma_0\neq \emptyset$ and $0<\varepsilon<1$, there exists $f\in A(K)$ and $t_0\in U\cap \Gamma_0$ such that $f(t_0)=\|f\|_\infty=1$, $|f(t)|<\varepsilon$ for every $t\in K\setminus U$, and for all $t\in U$, we have $|f(t)|+(1-\varepsilon)|1-f(t)|\leq 1$.
% \end{Lemma}

We are ready to state and prove the analogous result of Theorem \ref{theo-charac-Daug-C0LX} for $A(K,X)$ spaces. 

% This was known:
% \begin{itemize}
% \item In $A_u(K)$, Daugavet and $\Delta$ are the same as NA at $\Gamma'$.
% \item In $A(K,X)$, NA at $\Gamma'$ implies Daugavet. $\Delta$ implies NA at $\Gamma'$ if $X$ is uniformly rotund (but obviously not in general, since it doesn't in $C(K,X)$.
% \item In $A(K,X)$, diametral properties are equivalent if $X$ is uniformly rotund.
% \item There is an Urysohn-type lemma for $A_u(K)$.
% \item We use that characteristic functions of isolated points of $\Gamma$ are in $A_u(K)$ to split the space as a $\oplus_\infty$-sum. Same for vector-valued case. \textcolor{red}{This should be added to preliminaries or somewhere like here}.
% \item $A_u(K)$ in the real case is just $C(K)$. As for the complex case, it lacks strongly exposed points. Thus, in $A_u(K)$, $\nabla$ and Daugavet always coincide.
% \end{itemize}

% In this section we show:
% \begin{itemize}
% \item The analogous results of $C_0(L,X)$ also work on $A(K,X)$.
% \item In particular, in $A_u(K)$, all notions coincide.
% \end{itemize}

% The characterization \textcolor{red}{(well... or study)} of notions is almost identical to the case of $C_0(L,X)$.

% \textcolor{red}{Rewritten for $A(K,X)$ but not checked. We have to check if this is correctly written:}

\begin{theorem}\label{theo-charac-Daug-AKX}
Let $K$ be a compact Hausdorff topological space, let $X$ a Banach space, and let $f\in S_{A(K, X)}$. Let $\Gamma$ be the Shilov boundary of the base algebra $A(K)$.
\begin{enumerate}
\item If $f$ attains its norm at some non-isolated point $t_0\in \Gamma'$, then $f$ is a ccs-Daugavet point.
\item Otherwise, let $T:=\{t_1,\ldots,t_n\}\subset \Gamma$ ($n\in\bbn$) be the collection of isolated points $t_k$ of $\Gamma$ where $f$ attains its norm, which is finite and non-empty, and we have the following.
\begin{enumerate}
\item $f$ is super-Daugavet (resp. super-$\Delta$, or ccs-Daugavet) if and only if $f|_T$ is super-Daugavet (resp. super-$\Delta$, or ccs-Daugavet) in $A(T,X)=X\oplus_\infty\stackrel{n)}{\cdots}\oplus_\infty X$.
\item $f$ is Daugavet if and only if there is some $t_0\in T$ such that $f|_{\{t_0\}}$ is Daugavet in $A(\{t_0\}, X)=X$.
\item If $f$ is ccs-$\Delta$, then $f|_T$ is ccs-$\Delta$ in $A(T,X)=X\oplus_\infty\stackrel{n)}{\cdots}\oplus_\infty X$.
\item $f$ is $\Delta$ if and only if either there is some $t_k\in T$ such that $f|_{\{t_k\}}$ is $\Delta$ in $A(\{t_k\}, X)=X$, or for each $\{p_k\}_{k=1}^n\subset (1,+\infty)$ with $\sum_{k=1}^n \frac{1}{p_k}=1$, there exists some $k\in\{1,\ldots,n\}$ such that $f|_{\{t_k\}}$ is $\Delta_{p_k}$ in $A(\{t_{k}\}, X)=X$. 
\item If $\Gamma\neq T$, or if $A(K,X)$ is a complex Banach space, $f$ is $\nabla$ if and only if $f$ is Daugavet.
\item If $\Gamma=T$, then $f$ is $\nabla$ if and only if either $f|_{\{t_i\}}$ is $\nabla$ in $A(\{t_i\}, X)=X$ for all $1\leq i\leq n$, or there is some $t_0\in T$ such that $f|_{\{t_0\}}$ is Daugavet in $A(\{t_0\}, X)=X$ (and in this case, $f$ is Daugavet as well).
\end{enumerate}
\end{enumerate}
% \begin{enumerate}
% \item $f$ is super-Daugavet (resp. ccs-Daugavet) if and only if $f|_T$ is super-Daugavet (resp. ccs-Daugavet) in $A(T,X)=X\oplus_\infty\stackrel{n)}{\cdots}\oplus_\infty X$.
% \item $f$ is Daugavet if and only if there is some $t_0\in T$ such that $f|_{t_0}$ is Daugavet in $A(\{t_0\}, X)=X$.
% \item If $f$ is $\Delta$ (resp. super-$\Delta$ or ccs-$\Delta$), then $f|_T$ is $\Delta$ (resp. super-$\Delta$ or ccs-$\Delta$) in $A(T,X)=X\oplus_\infty\stackrel{n)}{\cdots}\oplus_\infty X$.
% \item If $\Gamma\neq T$, or if $A(K,X)$ is a complex Banach space, $f$ is $\nabla$ if and only if $f$ is Daugavet.
% \item If $\Gamma=T$, then $f$ is $\nabla$ if and only if either $f|_{t_i}$ is $\nabla$ in $A(\{t_i\}, X)=X$ for all $1\leq i\leq n$, or there is some $t_0\in T$ such that $f|_{t_0}$ is Daugavet in $A(\{t_0\}, X)=X$ (and in this case, $f$ is Daugavet as well).
% \end{enumerate}
% \end{enumerate}
\end{theorem}

\begin{proof}
We show \textup{(1)} with the help of Lemmas \ref{lemma:urysohn-type} and \ref{lemma:stolzing}. We will follow ideas based on the proof of \cite[Theorem 4.2]{MPZpre}, however, several steps and computations need to be done differently and carefully. The computations will be written in the context of complex Banach spaces, but they also work for the real case.

Let $t_0$ be a limit point of $\Gamma$ and $x^*\in S_{X^*}$ be such that $\re(x^*(f(t_0)))=\|f(t_0)\|=1$. Let $\Gamma_0$ be the Choquet boundary of the base algebra $A$. Let $C:=\sum_{i=1}^L \lambda_i S_i$ be a convex combination of slices of $B_{A(K,X)}$. For each $i\in\{1,\ldots,L\}$, pick a function $g_i\in S_i$. 

Note that $t_0$ is an accumulation point of $K$, and $K$ is compact, so arguing as in the proof of \cite[Theorem 4.2]{MPZpre}, there exists a sequence $\{U_n\}_{n=0}^{\infty}$ of open neighbourhoods of $t_0$ such that
\begin{itemize}
\item[1.] $U_0=K$,
\item[2.] $\overline{U_{n+1}}$ is a proper subset of $U_n$ for all $n\geq 0$,
\item[3.] $\re(x^*\circ f)|_{\overline{U_n}}\geq 1-\frac{1}{n}$ and $\|g_i|_{\overline{U_n}} - g_i(t_0)\| \leq \frac{1}{n}$ for all $i\in\{1,\ldots,L\}$ and all $n\geq 1$.
\end{itemize}

Moreover, note that $t_0$ is a limit point of $\Gamma=\overline{\Gamma_0}$, so since $K$ is normal, by Lemma \ref{lemma:urysohn-type}, there is a sequence of points $\{t_k\}_{k\in\bbn}\subset \Gamma_0$ and a subsequence $\{U_{n_k}\}_{k\in \bbn}$ of $\{U_n\}_{n\in\bbn}$ with $n_1>1$ such that
\begin{itemize}
\item[1.] $t_k$ converges to $t_0$,
\item[2.] for each $k\in\bbn$, there is an open neighbourhood of $t_k$, $V_k$, such that $\overline{V_k}\subset U_{n_k}\setminus \overline{U_{n_{k+1}}}$,
\end{itemize}
and for each $k\in\bbn$, there is $p_k\in B_{A}$ satisfying the following: 
\begin{itemize}
\item[1.] $|p_k(t)|<\frac{1}{n_k}$ for every $t\in K\setminus V_k$,
\item[2.] $p_k(t_k)=1$,
\item[3.] and for all $t\in V_k$, we have $|p_k(t)| + (1-\frac{1}{n_k}) \left|1-p_k(t)\right|\leq 1$.
\end{itemize}

Note that the sequence $\{p_k\}_{k\in\bbk}$ is uniformly bounded and converges pointwise to $0$, so it converges weakly to $0$ (see for instance the argument used in the proof of \cite[Theorem 4.2.(i)]{LT22}). We claim now that for each $k\in\bbn$ and each $i\in\{1,\ldots,L\}$, we have
\begin{equation}\label{Claim-Stolz}
\|g_i - (f(t_0)+g_i(t_0)) p_k\|_\infty \leq 1+\frac{2}{n_k} + 2R_{k},
\end{equation}
where $R_{k}$ is the positive number
$$R_{k}:=\sqrt{\left(\frac{1}{2\left(1-\frac{1}{n_k}\right)}\right)^2 -\left( \frac{1}{2} \right)^2}\quad \stackrel{k\to\infty}{\longrightarrow}\quad 0.$$
Indeed, on the one hand, note that for all $k\in\bbn$ and all $1\leq i\leq L$, if $t\in K\setminus V_k$, then the triangle inequality yields
$$\|g_i(t) - (f(t_0)+g_i(t_0))p_k(t)\|\leq \|g_i(t)\| + \|f(t_0)+g_i(t_0)\|\cdot|p_k(t)|\leq 1 + 2\frac{1}{n_k}.$$
On the other hand, note that for all $1\leq i\leq L$, if $t\in V_k$, we have
\begin{align*}
\|g_i(t) - (f(t_0)+g_i(t_0))p_k(t)\|&\leq \|g_i(t) - g_i(t_0)\| + \|g_i(t_0) - (f(t_0)+g_i(t_0))p_k(t)\|\\
&\leq \frac{1}{n_k} + \|g_i(t_0) - (f(t_0)+g_i(t_0))p_k(t)\|.
\end{align*}
Now, given $p_k(t)$, which is in the Stolz region $\operatorname{St}_{(1/n_k)}$, Lemma \ref{lemma:stolzing} shows that if $q_k(t)$ is the closest point to $p_k(t)$ in $[0,1]$, then $|p_k(t)-q_k(t)|\leq R_k$. Therefore,
\begin{align*}
\|g_i(t_0) - (f(t_0) &+ g_i(t_0))p_k(t)\| \\
&\leq
\|g_i(t_0)(1-q_k(t)) + (-f(t_0))q_k(t)\| + \|g_i(t_0) - f(t_0)\|\cdot |p_k(t)-q_k(t)|\\
&\leq 1 + 2R_k,
\end{align*}
since $g_i(t_0)$ and $-f(t_0)$ both belong to $B_X$. This shows that that \eqref{Claim-Stolz} holds. Therefore, for each $i\in\{1,\ldots,L\}$ and $k\in\bbn$, the function
$$g_{i,k}:=\frac{1}{1+\frac{2}{n_k} + 2R_{k}}(g_i - (f(t_0)+g_i(t_0))p_k)$$
belongs to $B_{A(K,X)}$, and the sequence $\{g_{i,k}\}_{k\in\bbn}$ converges weakly to $g_i$. Since the finitely many slices $S_i$ ($1\leq i\leq L$) are weakly open, we may find some $N\in \bbn$ such that for all $k\geq N$ and all $i\in\{1,\ldots,L\}$, $g_{i,k}\in S_i$. In particular, the function 
$$h_k:= \sum_{i=1}^L \lambda_i g_{i,k}$$
belongs to $C$ for all $k\geq N$. Fix $k\geq N$, and let us compute a lower bound for $\|f-h_k\|_\infty$. Consider the point $t_k\in V_k$. On the one hand, note that $\re(x^*(f(t_k)))\geq 1-\frac{1}{n_k}$. On the other hand, for each $i\in\{1,\ldots,L\}$,
\begin{align*}
\re(x^*(g_{i,k}(t_k)))&=\frac{1}{1+\frac{2}{n_k} + 2R_{k}} \re(x^*(g_i(t_k) - (f(t_0)+g_i(t_0))p_k(t_k)))\\
&=\frac{1}{1+\frac{2}{n_k} + R_{i,k}}\left( -1 + \re(x^*(g_i(t_k) - g_i(t_0))) \right)\\
&\leq \frac{1}{1+\frac{2}{n_k} + 2R_{k}} \left( -1 + \frac{1}{n_k} \right)= -1+\frac{2R_{k} + \frac{3}{n_k}}{1+\frac{2}{n_k}+2R_{k}}.
\end{align*}

Therefore, 
$$\|f-h_k\|_\infty \geq \re\left( x^*\left( f(t_k) - \sum_{i=1}^L \lambda_i g_{i,k}(t_k) \right)\right)\geq 2-\frac{1}{n_k} - \frac{2R_{k} + \frac{3}{n_k}}{1+\frac{2}{n_k}+2R_{k}}\stackrel{k\to\infty}{\longrightarrow}2.$$
Since this can be done for each $k\geq N$, we conclude that $f$ is a ccs-Daugavet point.

Finally, to see \textup{(2)} it suffices to argue as in the proof of Theorem \ref{theo-charac-Daug-C0LX}.(2).
\end{proof}

% \textcolor{red}{We need to rethink how to write the section, since the proof of item (1) above is done later in the implications theorem, as it uses the following lemmas. The proof of all other items should be an easy adaptation to that of $C_0(L,X)$, hopefully. Suggestion: 1st lemmas, then characterization, then implications.}

% Regarding the analogous result of Theorem \ref{theo:implications-C0LX} for $A(K,X)$ spaces, note that the proof of \cite[Theorem 4.2]{MPZpre} relies heavily on Urysohn's Lemma, which is not true in general in uniform algebras. However, there is an Urysohn-type lemma from \cite{CGK13} that we can carefully apply.

Once more, it is not true in general that these notions coincide in $A(K,X)$ spaces, but the situation is different if $X$ is LUR. In fact, we have the following.

\begin{Theorem}\label{theo:6notions-AKX}
Let $K$ be a compact Hausdorff topological space and let $X$ be a Banach space. Let $\Gamma$ be the Shilov boundary of the base algebra $A$ of $A(K, X)$. Given $f\in S_{A(K,X)}$, consider the following diagram of implications
\[
\begin{tikzcd}
    \begin{array}{c} \|f(t)\|=\|f\|_\infty\\ \text{ for some }t\in \Gamma' \end{array} \\
    {f\text{ ccs-Daugavet}} & {f\text{ super-Daugavet}} & {f\text{ Daugavet}} & {f\text{ }\nabla} \\
    {f\text{ ccs-}\Delta} & {f\text{ super-}\Delta} & {f\text{ }\Delta} & \begin{array}{c} \|f(t)\|=\|f\|_\infty\\ \text{ for some }t\in \Gamma' \end{array}
    \arrow["{1)}", from=1-1, to=2-1]
    \arrow[from=2-1, to=2-2]
    \arrow[from=2-1, to=3-1]
    \arrow[from=2-2, to=2-3]
    \arrow[from=2-2, to=3-2]
    \arrow[from=2-3, to=2-4]
    \arrow[from=2-3, to=3-3]
    \arrow["{4)}", from=2-3, to=3-4]
    \arrow["{3)}", from=2-4, to=3-4]
    \arrow[bend right=20, from=3-1, to=3-3]
    \arrow[from=3-2, to=3-3]
    \arrow["{2)}"', from=3-3, to=3-4]
\end{tikzcd}
\]
Then:
\begin{enumerate}
\item[(i)] The unnumbered implications are always true in every Banach space.
\item[(ii)] The implication \textup{(1)} always holds in $A(K,X)$.
\item[(iii)] While the implication \textup{(2)} is not always true in general, it is true if $X$ is LUR.
\item[(iv)] The implication \textup{(3)} holds for every $f\in S_{A(K,X)}$ if and only if either $\Gamma=\Gamma'$ or $X$ has no $\nabla$ points.
\item[(v)] The implication \textup{(4)} holds for every $f\in S_{A(K,X)}$ if and only if either $\Gamma=\Gamma'$ or $X$ has no Daugavet points.
\end{enumerate}
\end{Theorem}

\begin{proof}
\textup{(i)} is well known, and \textup{(ii)} has been shown in the proof of Theorem \ref{theo-charac-Daug-AKX}. To show \textup{(iii)} when $X$ is LUR, it suffices to follow verbatim the proof of \cite[Theorem 5.4]{LT22}, but instead of taking $\delta=\delta_X(1/|F|)$, we take 
$$\delta=\inf_{t\in F} \delta_X\left(f(t), \frac{1}{|F|} \right),$$
where 
$$\delta_X(x, \varepsilon):=\inf\left\{ 1-\left\|\frac{x+y}{2}\right\|:\, y\in B_X,\, \|x-y\|\geq\varepsilon \right\}$$
is the modulus of local rotundity at $x$.

Finally, \textup{(iv)} and \textup{(v)} can be derived from the stability results from Section \ref{section:stability} such as in the proof of Theorem \ref{theo:implications-C0LX}.
\end{proof}

As a consequence, we get that all 6 notions are equivalent in uniform algebras, which extends the conclusions from \cite[Corollary 5.5]{LT22}.

\begin{Corollary}\label{cor:charac-unif-algebras}
Let $K$ be a compact Hausdorff topological space, $\Gamma$ be the Shilov boundary of $A(K)$, and $f\in S_{A(K)}$. Then the following statements are equivalent:
\begin{enumerate}
\item There is a limit point $t_0$ of $\Gamma$ such that $\|f\|_\infty=|f(t_0)|$. 
\item $f$ is a ccs Daugavet point.
\item $f$ is a ccs $\Delta$ point.
\item $f$ is a super Daugavet point.
\item $f$ is a super $\Delta$ point.
\item $f$ is a Daugavet point.
\item $f$ is a $\Delta$ point.
% \item $f$ is a $\nabla$ point.
\end{enumerate}
\end{Corollary}

In fact, $\nabla$ points are also equivalent to the other notions in all uniform algebras $A(K)$ except in the real and finite-dimensional case.

Another consequence of Theorem \ref{theo:6notions-AKX} is the following improvement of \cite[Proposition 4.15]{LT24}.

\begin{Corollary}\label{cor:akxequiv}
Let $K$ be a compact Hausdorff topological space, $X$ be LUR, and $\Gamma$ be the Shilov boundary for the base algebra of $A(K, X)$. Then the following statements are equivalent:
\begin{enumerate}
\item $\Gamma$ does not contain isolated points.
\item The space $A(K, X)$ has the polynomial Daugavet property.
\item The space $A(K, X)$ has the restricted DSD2P.
\item The space $A(K, X)$ has the DD2P.
\item The space $A(K, X)$ has the DLD2P.
\item The space $A(K, X)$ has the property ($\mathcal{D}$).
\end{enumerate}
\end{Corollary}

\subsection{Results on \texorpdfstring{$\boldsymbol{L_\infty(\mu, X)}$}{L\textinfty(mu,X)} spaces}\hfill\\ \label{subsection:Linfty(mu,X)}

In this section, we will study the notions in $L_\infty(\mu, X)$ spaces, where $(\Omega, \Sigma, \mu)$ is a measure space. Recall that if $\mu$ is a measure, then $L_\infty(\mu)$ is isometric to $C(K)$ for some compact $K$, and thus in these spaces all diametral point notions (including $\nabla$ except in the real and finite-dimensional case) are equivalent. However, to the authors' knowledge, there are no known conditions that characterize these points in $L_\infty(\mu)$ spaces. Furthermore, much less is known about the vector-valued $L_\infty(\mu, X)$ spaces, where the situation is often different from the scalar-valued case. 

In \cite[Theorem 2.5]{MV03}, it was shown that if $\mu$ is $\sigma$-finite, then the space $L_\infty(\mu, X)$ has Daugavet property if and only if either $\mu$ is atomless or $X$ has Daugavet property. Here we will provide a characterization of all the $\Delta$, Daugavet, and $\nabla$ points in $L_\infty(\mu, X)$ spaces, as well as some conditions regarding the other notions. One of the consequences of our results is that \cite[Theorem 2.5]{MV03} still holds even when $\mu$ is an arbitrary measure, where the forward implication comes from Lemma \ref{lemma-Daugavet-Down-infty-sum}.

First we characterize Daugavet points in $L_\infty(\mu, X)$ spaces.

\begin{theorem}\label{theo:charac-Daug-points-Linfty}
Given $(\Omega, \Sigma, \mu)$ a measure space and $X$ a Banach space, let $f\in S_{L_\infty(\mu, X)}$. 
\begin{enumerate}
\item If for every $\varepsilon>0$ there is an infinite collection of disjoint sets $\{B_n\}_{n=1}^\infty$ of positive measure such that
$$\|f(t)\|_X>1-\frac{\varepsilon}{2} \text{ for each }t\in B_n,\text{ and each }n\in\bbn,$$
then $f$ is a Daugavet point.
\item Otherwise, the collection of atoms $\{A_j\}_{j\in J}$ such that $\|f|_{A_j}\|=1$ for all $j\in J$ is finite and non-empty, and $f$ is Daugavet if and only if there is $j_0\in J$ such that $f|_{A_{j_0}}$ is Daugavet in $L_\infty(\mu|_{A_{j_0}}, X)=X$.
\end{enumerate}
\end{theorem}

\begin{proof}
\textup{(1)}: We follow the steps of the proof of \cite[Theorem 2.5]{MV03}. Recall that $f$ is a Daugavet point if and only if for every $\varepsilon>0$, $B_{L_\infty(\mu, X)}= \overline{\co}(\Delta_\varepsilon(f))$. Let $\varepsilon>0$ and let $\{B_n\}_{n=1}^{\infty}$ be a family of pairwise disjoint measurable subsets of $\Omega$. Fix $k\in\bbn$, and choose $k$ distinct elements $C_1,\ldots,C_k$ of  $\{B_n\}_{n=1}^{\infty}$. Fix any $h\in B_{L_\infty(\mu, X)}$. For each $1\leq i\leq k$, define $g_i:= h\cdot \chi_{\Omega\setminus C_i} - f\cdot \chi_{C_i}$. If $t\in C_i$ for some $1\leq i\leq k$, then
$$\left\| h(t) - \frac{1}{k}\sum_{j=1}^k g_j(t)\right\|_X\leq \frac{1}{k}\|h(t)+f(t)\|_X\leq \frac{2}{k},$$
and if $t\notin \bigcup_{i=1}^k C_i$, then $h(t) = \frac{1}{k}\sum_{j=1}^k g_j(t)$. Finally, note that for all $1\leq i\leq k$, $\|f-g_i\|>2-\varepsilon$, so $g_i\in \Delta_\varepsilon(f)$. This shows that $f$ is a Daugavet point.

\textup{(2)}: Note that if the condition of \textup{(1)} is not satisfied, there is no atomless set $S\subset \Omega$ such that $\|f|_S\|=1$, and the collection $\{A_j\}_{j\in J}$ cannot be infinite or empty, so $0<|J|<\infty$. Let us denote $B=\bigcup_{j\in J} A_j$. Then we can split $L_\infty(\mu, X)=L_\infty(\mu|_{\Omega\setminus B}, X)\oplus_\infty L_\infty(\mu|_B, X)$ and $L_\infty(\mu|_B, X)=X\oplus_\infty \stackrel{|J|}{\cdots}\oplus_\infty X$. Hence, by Lemma \ref{daugavet-infty-sums} we obtain the result. 
\end{proof}

As a consequence, we have the following result.

\begin{corollary}[{see \cite[Theorem 2.5]{MV03} for $\sigma$-finite measures}]\label{cor:Linfty-Daugavet-Property}
Given a measure space $(\Omega, \Sigma, \mu)$ and $X$ a Banach space, the space $L_\infty(\mu, X)$ has Daugavet property if and only if either $\mu$ is atomless or $X$ has Daugavet property.
\end{corollary}

\begin{proof}
If $\mu$ is atomless, then every point $f\in S_{L_\infty(\mu, X)}$ satisfies the condition of Theorem \ref{theo:charac-Daug-points-Linfty}.\textup{(1)}. Hence $L_\infty(\mu, X)$ has Daugavet property. For the remaining implications it suffices to apply Theorem \ref{theo:charac-Daug-points-Linfty}.\textup{(2)} together with Lemma \ref{lemma-Daugavet-Down-infty-sum}.
\end{proof}

%Note that the previous results provide in turn easy conditions characterizing the notions in the scalar-valued $L_\infty(\mu)$ spaces. For instance, we have the following for $\ell_\infty$.

\begin{example}\label{example:ell_infty}
Theorem \ref{theo:charac-Daug-points-Linfty}.(1) provides an easy characterization of Daugavet points in $L_\infty(\mu)$ spaces (and thus, of every other notion as well). In particular, a point $x=(x_1, x_2, \ldots)\in S_{\ell_\infty}$ is a ccs-Daugavet point (equivalently, a $\Delta$ point) if and only if $\limsup_{n\in\bbn} |x_n|=1$.
\end{example}

Now, we can study all the notions as we did in $C_0(L,X)$ and $A(K,X)$ spaces.

% \textcolor{red}{To fix notation: We probably can't just put $\cala$ as the union of all atoms without assuming some conditions, so we have to detour it by sticking to $\supp(f)$. This will be fixed in the next version:}

\begin{theorem}\label{theo:Linfty-All-Together}
Let $(\Omega, \Sigma, \mu)$ be a measure space, let $X$ be a Banach space, and let $f\in S_{L_\infty(\mu, X)}$. Suppose that $\Omega$ contains an atom $A_0$ such that $\mu(\Omega\setminus A_0)>0$, and denote the distinct atoms by $\{A_i:\, i\in I\}$ for some index set $I$.
% , and denote $\cala=\bigcup_{i\in I} A_i$ to the union all the distinct atoms $A_i$ of $\Omega$.
\begin{enumerate}
\item If there is an atomless set $S\subset \supp(f)$ such that $\|f|_S\|=1$, then $f$ is a ccs-Daugavet point.
\item If for every $\varepsilon>0$ there is an infinite collection of distinct atoms $\{A_\lambda\}_{\lambda\in \Lambda}$ such that for every $\lambda\in \Lambda$ we have $\|f|_{A_{\lambda}}\|>1-\frac{\varepsilon}{2}$, then $f$ is a Daugavet point. 
\item Otherwise, let $J:=\{i\in I:\, \|f|_{A_i}\|=1\}$, which is finite and non-empty, and let $\calb=\bigcup_{j\in J} A_j$. We have the following.
\begin{enumerate}
\item $f$ is super-Daugavet (resp. super-$\Delta$, or ccs-Daugavet) if and only if $f|_\calb$ is super-Daugavet (resp. super-$\Delta$, or ccs-Daugavet) in $L_\infty(\mu|_\calb,X)=X\oplus_\infty\stackrel{n)}{\cdots}\oplus_\infty X$.
\item $f$ is Daugavet if and only if there is some atom $A_0\subset \calb$ such that $f|_{A_0}$ is Daugavet in $L_\infty(\mu|_{A_0}, X)=X$.
\item If $f$ is ccs-$\Delta$, then $f|_\calb$ is ccs-$\Delta$ in $L_\infty(\mu|_\calb,X)=X\oplus_\infty\stackrel{n)}{\cdots}\oplus_\infty X$. %and $\mu(\Omega\setminus \calb)>0$
\item $f$ is $\Delta$ if and only if either there is some atom $A_0\subset\calb$ such that $f|_{A_0}$ is $\Delta$ in $L_\infty(\mu|_{A_0}, X)=X$, or for each $\{p_k\}_{k=1}^{|J|}\subset (1,+\infty)$ with $\sum_{k=1}^{|J|} \frac{1}{p_k}=1$, there exists some $k\in\{1,\ldots,|J|\}$ such that $f|_{A_k}$ is $\Delta_{p_k}$ in $L_\infty(\mu|_{\{A_k\}}, X)=X$. 
\item If $\mu(\Omega\setminus \calb)>0$, or if $L_\infty(\mu,X)$ is a complex Banach space, $f$ is $\nabla$ if and only if $f$ is Daugavet.
\item If $\Omega=\calb$, then $f$ is $\nabla$ if and only if either $f|_{A_i}$ is $\nabla$ in $L_\infty(\mu|_{A_i}, X)=X$ for all $1\leq i\leq n$, or there is some atom $A_0\subset \calb$ such that $f|_{A_0}$ is Daugavet in $L_\infty(\mu|_{A_0}, X)=X$ (and in this case, $f$ is Daugavet as well).
\end{enumerate}
\end{enumerate}
\end{theorem}

\begin{proof}
To show \textup{(1)}, just note that we can split $f=(f|_{S}, f|_{\Omega \setminus S})$ in $L_\infty(\mu|_{S}, X)\oplus_\infty L_\infty(\mu|_{\Omega\setminus S}, X)$. Since $L_\infty(\mu|_{S}, X)$ has Daugavet property by Corollary \ref{cor:Linfty-Daugavet-Property}, $f|_{S}\in L_\infty(\mu|_{S}, X)$ is a ccs-Daugavet point, and so, $f$ is also ccs-Daugavet by \cite[Theorem 3.32]{MPZpre}.

Note that \textup{(2)} is already shown in Theorem \ref{theo:charac-Daug-points-Linfty}, so only \textup{(3)} remains, and this is once more a consequence of applying the stability lemmas from Section \ref{section:stability}.
\end{proof}

\begin{remark}
Theorem \ref{theo:Linfty-All-Together} implies that the $\Delta$ and $\nabla$ points exist and coincide for every infinite-dimensional $L_\infty(\mu)$ space. However, the vector-valued scenario is completely different. For instance, there are infinite-dimensional $L_\infty(\mu, X)$ spaces with $\Delta$ points that are not $\nabla$ (e.g. $\Omega$ is just 2 atoms and $X=C([0,1])\oplus_2 C([0,1])$) or without either type of points (e.g. $\Omega$ is just 2 atoms and $X=c_0$).
\end{remark}

\subsection{Results on \texorpdfstring{$\boldsymbol{L_1(\mu, X)}$}{L1(mu,X)} spaces}\hfill\\ \label{subsection:L1(mu,X)}

In this section, we will study $L_1(\mu, X)$ spaces, where $(\Omega, \Sigma, \mu)$ will always denote a measure space. It is well-known that in $L_1(\mu)$, the notions of $\Delta$, Daugavet, super-$\Delta$, and super-Daugavet (and of course $\nabla$ in the complex case) points are equivalent. Moreover, these notions can be characterized by the support being atomless (see \cite[Proposition 4.7 and Theorem 4.8]{MPZpre}). In the real case, it is also shown that ccs-$\Delta$ points also coincide with those notions (\cite[Proposition 4.9]{MPZpre}), and that $f\in S_{L_1(\mu)}$ is $\nabla$ if and only if its support is either atomless or a single atom (\cite[Proposition 3.7]{HLPV}). However, these notions do not coincide with ccs-Daugavet because $L_1(\mu)$ contains ccs-Daugavet points if and only if it has Daugavet property (\cite[Proposition 4.12]{MPZpre}).

The situation is different in the vector-valued case: for $f\in S_{L_1(\mu, X)}$, having atomless support implies being super-Daugavet (\cite[Theorem 4.8]{MPZpre}), but being $\Delta$ does not imply this condition or being Daugavet in general (see Remark \ref{remark:uses-of-stability}). In this section we will characterize $\Delta$, Daugavet, and $\nabla$ points in $L_1(\mu, X)$ spaces and study the other notions. Let us start with an improvement of the result in the scalar-valued case.

% We begin by extending a result in the scalar-valued $L_1(\mu)$ spaces.

% \textcolor{red}{This was known:}

% \begin{itemize}
% \item In $L_1(\mu, X)$, atomless support implies super-Daugavet.
% \item In the scalar case, in $L_1(\mu)$ that condition characterizes from super-Daugavet to $\Delta$ points. Also ccs-$\Delta$ points in the real case, complex case unknown.
% \item In the real case, $f\in S_{L_1(\mu)}$ is $\nabla$ if and only if its support is either atomless or a single atom.
% \end{itemize}

% Also, this was not explicitly written, but it is clear by arguing as in \cite[Remark 4.4]{MPZpre}:
% \begin{itemize}
% \item From the stability results we can conclude once more that in $L_1(\mu, X)$, in general it is not true that $\Delta$ implies Daugavet, as we will see.
% \end{itemize}

% In this section, $(\Omega, \Sigma, \mu)$ will always denote a measure space with $\mu$ positive once more. 

As mentioned, the \textit{real} $L_1(\mu)$ spaces, atomless support implies being ccs-$\Delta$ \cite[Proposition 4.9]{MPZpre}. This proof relies on a result from \cite[Theorem]{ABHLP20} whose proof only works in the real case. However, we can show further by Lemma \ref{ccs-Delta-1-sums}.

\begin{theorem}[{Extension of \cite[Proposition 4.9]{MPZpre}}]\label{theorem:ccs-delta-L1}
Let $(\Omega, \Sigma, \mu)$ be an arbitrary measure space, and let $X$ be a Banach space. If $f\in S_{L_1(\mu, X)}$ has atomless support, then it is a ccs-$\Delta$ point.
\end{theorem}

\begin{proof}
% We argue as in the proof of \cite[Theorem 4.8]{MPZpre}. 
Write $S:=\supp(f)$. Since $S$ is atomless, $L_1(\mu|_S, X)$ has Daugavet property, so $f|_S$ is a ccs-$\Delta$ point in this space. Then, by Lemma \ref{ccs-Delta-1-sums}, $f=(f|_S, 0)$ is a ccs-$\Delta$ point in $L_1(\mu, X)=L_1(\mu|_S, X)\oplus_1 L_1(\mu|_{\Omega\setminus S}, X)$.
\end{proof}

In particular, we can characterize ccs-$\Delta$ points in $L_1(\mu)$ spaces even in the complex case.

\begin{corollary}[{\cite[Corollary 4.11]{MPZpre} for the real case}]\label{cor:notions-L1(mu)}
Let $(\Omega, \Sigma, \mu)$ be an arbitrary measure space and let $f\in S_{L_1(\mu)}$. Then the following claims are equivalent: $\supp(f)$ is atomless, $f$ is a ccs-$\Delta$ point, $f$ is a super-Daugavet point, $f$ is a super-$\Delta$ point, $f$ is a Daugavet point, $f$ is a $\Delta$ point.
\end{corollary}

Now we will characterize $\Delta$, Daugavet, and $\nabla$ points in $L_1(\mu, X)$ spaces.

\begin{theorem}\label{charac-notions-L1(mu,X)}
Let $(\Omega, \Sigma, \mu)$ be any measure space, let $X$ be a Banach space, and let $f\in S_{L_1(\mu, X)}$.
\begin{enumerate}
\item If $\supp(f)$ is atomless, then $f$ is a super-Daugavet and ccs-$\Delta$ point.
\item Otherwise, let $\mathcal{A}:=\bigcup_{i\in I} A_i$ be the union of all distinct atoms $A_i$ in the support of $f$. Then the following holds.
\begin{enumerate}
\item $f$ is $\Delta$ (respectively Daugavet) if and only if for every $i\in I$, the function $\frac{f|_{A_i}}{\|f|_{A_i}\|}$ is $\Delta$ (respectively Daugavet) in $L_1(\mu|_{A_i}, X)=X$.
\item In the real case, $f$ is $\nabla$ if and only if either $f$ is Daugavet or there is an atom $A$ such that $f=(0, f|_A)\in L_1(\mu|_{\Omega\setminus A}, X)\oplus L_1(\mu|_A, X)$ and $f|_A$ is a $\nabla$ point in $L_1(\mu|_A, X)=X$.
\end{enumerate}
\end{enumerate}
\end{theorem}

\begin{proof}
% \textcolor{red}{To rewrite/fix: we can't split like that, we need to detour and use $\supp(f)$. So we take $A$ as atoms in $\supp(f)$, $S=\supp(f)\setminus A$, and $B$ is whatever remains.}

Note that \textup{(1)} is shown in \cite[Theorem 4.8]{MPZpre} and Theorem \ref{theorem:ccs-delta-L1}. Also, \textup{(2).(b)} is deduced from Lemma \ref{nabla-1-sums}. % \cite[Propositions 3.3 and 3.4]{HLPV} \textcolor{red}{(or maybe better cite our own versions of stability lemmas)}. 

Finally, we show \textup{(2).(a)}. Let $\mathcal{A}\neq \emptyset$ be given. Since the image of $f$ is essentially separable, $I$ is at most countable. Let $S=\supp(f)\setminus \cala$, and let $\calb:=\Omega\setminus \supp(f)$. Note that we can split $f=\left( f|_S, 0, \|f|_{\mathcal{A}}\| \frac{f|_{\mathcal{A}}}{\|f|_{\mathcal{A}}\|} \right)$ in
$$L_1(\mu, X)=L_1(\mu|_S, X)\oplus_1 L_1(\mu|_{\calb}, X)\oplus_1 L_1(\mu|_\cala, X),$$
where $L_1(\mu|_\cala, X)$ is $\ell_1(X):=\left[ \bigoplus_{n=1}^{\infty} X \right]_{\ell_1}$ if $I$ is infinite and $\left[ \bigoplus_{n=1}^{|I|} X \right]_{\ell_1}$ otherwise.

Finally, if $S\neq \emptyset$, then $f|_S=\|f|_S\| \frac{f|_S}{\|f|_S\|}$, and $\frac{f|_S}{\|f|_S\|}$ is a super-Daugavet point in $L_1(\mu|_S, X)$ (in particular a Daugavet and a $\Delta$ point), so the rest of claims follow from applying Lemmas \ref{delta-daug-1-sums} and \ref{delta-daug-ell1-countable-sums}.
\end{proof}

\begin{remark}
    In \textit{real} $L_1(\mu)$ spaces, the $\nabla$ points are characterized by their supports being either atomless or a single atom. Theorem \ref{charac-notions-L1(mu,X)} shows that the situation is quite different for the vector-valued case. Indeed, on the one hand, if $X$ contains no Daugavet points, then a point being $\nabla$ implies the condition on its support, but the converse statement does not necessarily hold, because there may be functions supported on an atom that are not $\nabla$.
    On the other hand, the condition on the support is not necessary in some cases. Indeed, if $X$ contains a Daugavet point and $\Omega$ contains at least $2$ disjoint sets of positive measure, then there are always $\nabla$ points whose support is neither atomless nor a single atom.
\end{remark}

% \textcolor{red}{Remarks:}

% \color{red}
% \begin{remarks}
% As a consequence of our results, the notions of $\Delta$ and Daugavet can also be separated in some $L_1(\mu, X)$ spaces, and they do not necessarily imply atomless support, unlike in the scalar-valued counterpart. As for real $\nabla$ points, they also differ from the scalar case:
% \begin{itemize}
% \item If $X$ contains no Daugavet points, then $\nabla$ implies that the support is either atomless or a single atom. If the support is atomless, then the function is $\nabla$. But there may be functions supported on a single atom which are not $\nabla$. Take for instance $\Omega$ consisting of just two atoms and $X$ to be a space without $\nabla$ points such as $c_0$ or any LUR space and apply \textcolor{red}{the stability results}.
% \item If $X$ contains a Daugavet point and $\Omega$ contains at least 2 disjoint sets of positive measure, then there are always $\nabla$ points whose support is neither atomless nor a single atom. Indeed, just consider $f_1\in S_{L_1(\mu, X)}$ with atomless support, $f_2\in S_{L_1(\mu, X)}$ supported on a single atom $A$ and such that $f_2|_A$ is a Daugavet point in $L_1(\mu|_A, X)=X$, and note that $f=(\frac{1}{2}f_1, \frac{1}{2}f_2)$ is a Daugavet point in $L_1(\mu, X)=L_1(\mu|_{\Omega\setminus A}, X)\oplus_1 L_1(\mu|_A, X)$ by \textcolor{red}{the stability results}.
% \end{itemize}
% \end{remarks}\color{black}

Note that we can also characterize when being $\Delta$ or Daugavet implies that the support is atomless, or when $\Delta$ implies Daugavet.

\begin{proposition}
Let $(\Omega, \Sigma, \mu)$ be a measure space, and let $X$ be a Banach space. Let $f\in S_{L_1(\mu, X)}$. Consider the following statements.
\begin{enumerate}
\item The support of $f$ contains no atoms. 
\item $f$ is a ccs-$\Delta$ point
\item $f$ is a super Daugavet point.
\item $f$ is a super $\Delta$ point.
\item $f$ is a Daugavet point.
\item $f$ is a $\Delta$ point.
\end{enumerate}
Then \textup{(1)} implies every other item, and every item implies \textup{(6)}. Moreover,
\begin{itemize}
\item[(a)] If $\mu$ is atomless, then \textup{(6)} implies \textup{(1)}.
\item[(b)] If $\mu$ contains an atom, then:
\begin{itemize}
\item[(i)] If $S_X$ contains no $\Delta$ points, then \textup{(6)} implies \textup{(1)}. Otherwise:
\item[(ii)] If $S_X$ contains a $\Delta$ point $x$, then \textup{(6)} does not imply \textup{(1)}. 
\item[(iii)] If $S_X$ contains a $\Delta$ point $x$ that is not Daugavet, then \textup{(6)} does not imply \textup{(5)}. 
\item[(iv)] If $S_X$ contains a Daugavet point $x$, then \textup{(5)} does not imply \textup{(1)}. 
\end{itemize}
\end{itemize}
\end{proposition}

\begin{proof}
Note that in \textup{(a)} the space has Daugavet property, and in \textup{(b)} the space can be split as $L_1(\mu, X)=X\oplus_1 Y$ for some Banach space $Y$, and now the conclusion follows from Lemma \ref{delta-daug-1-sums} and Theorem \ref{charac-notions-L1(mu,X)}.
\end{proof}

\begin{Remark}
There are classes of Banach spaces with no $\Delta$ points. Some examples of such spaces are asymptotically uniformly smooth Banach spaces, asymptotically uniformly convex Banach spaces with reflexivity \cite{ALMP22}, and locally uniformly nonsquare Banach spaces \cite{KLT24, LT24} (which includes LUR Banach spaces). On the other hand, there exists a reflexive Banach space with super-$\Delta$ points \cite[Theorem 3.1]{AALMPPV24pre}.    
\end{Remark}

We make a comment on ccs-Daugavet points now. Recall that $S_{L_1(\mu)}$ contains a ccs-Daugavet point if and only if $L_1(\mu)$ has Daugavet property, and in that case every point is ccs-Daugavet. We have the following.

\begin{theorem}
Let $(\Omega, \Sigma, \mu)$ be any measure space, and let $X$ be a Banach space containing a denting point. Then the following claims are equivalent:
\begin{enumerate}
\item $L_1(\mu)$ contains a ccs-Daugavet point.
\item $L_1(\mu, X)$ contains a ccs-Daugavet point.
\item $\mu$ is atomless.
\end{enumerate}
\end{theorem}

\begin{proof}
The equivalence between \textup{(1)} and \textup{(3)} was established in \cite[Proposition 4.12]{MPZpre}.

\textup{(3)} implies \textup{(2)}: If $\mu$ is atomless, then $L_1(\mu, X)$ has Daugavet property (see for instance \cite[Theorem 4.8]{MPZpre}), so every point is ccs-Daugavet.

\textup{(2)} implies \textup{(3)}: When $\Omega$ consists of a single atom, $L_1(\mu, X)=X$ cannot contain ccs-Daugavet points (because if a space has a ccs-Daugavet point, in particular all slices have diameter 2, so there cannot be denting points, which is a contradiction to our assumption). So we only need to prove this for the case where $\Omega$ contains at least two disjoint subsets of positive measure. If $\mu$ contains an atom, then we can split $L_1(\mu, X)=X\oplus_1 Y$ for some Banach space $Y$. Since $X$ has a denting point, $L_1(\mu, X)$ cannot have a ccs-Daugavet point in view of Proposition \ref{prop:denting-no-ccs-1-sums}.
\end{proof}

\begin{remark}
The condition on $X$ containing a denting point cannot be dropped in general. Indeed, consider a measure space $\Omega$ consisting of two distinct atoms of measure $1$ and $X$ having Daugavet property. Then $L_1(\mu, X)=X\oplus_1 X$ has Daugavet property (see \cite[Theorem 1]{Wojtaszczyk92}), so it contains ccs-Daugavet points, but $\mu$ is purely atomic, and so, $L_1(\mu)$ contains no ccs-Daugavet points (see \cite[Proposition 4.12]{MPZpre}).
\end{remark}

%Finally, note the following.

% \textcolor{red}{To write better, to add cites:}

\begin{remark}
Let $1<p<\infty$, let $(\Omega, \Sigma, \mu)$ be any measure space, and let $X$ be a Banach space. Note that, except for the trivial case where $\Omega$ consists of a single atom, you can always split $L_p(\mu, X)=L_p(\mu|_{\Omega_1}, X)\oplus_p L_p(\mu|_{\Omega_2}, X)$ for some nonempty disjoint subsets $\Omega_1,\Omega_2\subset \Omega$ where $\Omega_1 \cup \Omega_2 = \Omega$. Thus, these spaces cannot have $\nabla$, Daugavet, super-Daugavet, or ccs-Daugavet points (see \cite[Theorem 3.2]{HLPV} and \cite[Remark 4.5 and Proposition 4.6]{AHLP}). However, $\Delta$-type points on these vector-valued spaces differ from the scalar-valued ones because $L_p(\mu)$ cannot have $\Delta$ points due to uniform convexity. We distinguish two cases:
\begin{enumerate}
\item If $\mu$ contains an atom, then we can split $L_p(\mu, X)=X\oplus_p Y$ for some Banach space $Y$, and now if $X$ has $\Delta$, super-$\Delta$, or ccs-$\Delta$ points, then so does $L_p(\mu, X)$ by the stability lemmas.
\item If $\mu$ is atomless, if $X$ is LUR, then so is $L_p(\mu, X)$ \cite[Thoerem 2]{ST80} (also see the discussion in \cite[pp. 666]{KT89} for general $\sigma$-finite measure spaces). In this case, the space cannot contain $\Delta$ points. However, if $X$ has the DLD2P, then so does $L_p(\mu, X)$ \cite[Corollary 4.2]{LP21}. Hence, every point on $S_{L_p(\mu, X)}$ is $\Delta$.
\end{enumerate}
\end{remark}

% Consider the Banach space $Y=L_p(\mu, X)$. We distinguish two cases.
% \begin{itemize}
% \item If $\mu$ contains atoms, then we can split $Y=X\oplus_p Z$ for some $Z$. Thus, by the stability results, this space cannot contain ccs-Daugavet, super-Daugavet, Daugavet, or $\nabla$ points. However, if $X$ has $\Delta$, super-$\Delta$, or ccs-$\Delta$ points, then so does $Y$. In fact, it can happen that every point is $\Delta$ (consider for instance $\mu$ consisting of two atoms, $X=C([0,1])$, and $p=2$; then $L_2(\mu, X)=X\oplus_2 X$ has the DLD2P).
% \item Otherwise, we can no longer claim that $Y=X\oplus_p Z$. In this case, we note however that if $X$ has the DLD2P, then so does $L_p(\mu, X)$, and if $X$ is LUR (thus lacking $\Delta$ and $\nabla$ points), so is $L_p(\mu, X)$. \textcolor{red}{CHECK: Can we split $L_p(\mu, X) = L_p(\mu_1, X) \oplus_p L_p(\mu_2, X)$ in general? IF we can, then $L_p(\mu, X)$ will NEVER have $\nabla$, Daugavet, super-Daugavet, and ccs-Daugavet points. This is not clear.}
% \end{itemize}
% \end{remark}

\end{document}